\documentclass{article}
\usepackage{amsthm,amsmath,amssymb}

\numberwithin{equation}{section}
\theoremstyle{plain}

\newtheorem{theorem}{Theorem}
\newtheorem{lemma}{Lemma}
\newtheorem{corollary}{Corollary}
\newtheorem{Question}{Question}
\newtheorem{definition}{Definition}
\newtheorem{Example}{Example}
\newtheorem{remark}{Remark}

\newtheorem{Notation}{Notation}
\newtheorem{Observation}{Observation}

\DeclareMathOperator{\pois}{Pois}

\allowdisplaybreaks

\title{A Development of Continuous-Time Transfer Entropy
}

\author{Joshua N. Cooper         \and
        Christopher D. Edgar 
}

\begin{document}

\maketitle

\begin{abstract}
Transfer entropy (TE) was introduced by Schreiber in 2000 as a measurement of the predictive capacity of one stochastic process with respect to another.  Originally stated for discrete time processes, we expand the theory in line with recent work of Spinney, Prokopenko, and Lizier to define TE for stochastic processes indexed over a compact interval taking values in a Polish state space. We provide a definition for continuous time TE using the Radon-Nikodym Theorem, random measures, and projective limits of probability spaces.  As our main result, we provide necessary and sufficient conditions to obtain this definition as a limit of discrete time TE, as well as illustrate its application via an example involving Poisson point processes. As a derivative of continuous time TE, we also define the transfer entropy rate between two processes and show that (under mild assumptions) their stationarity implies a constant rate. We also investigate TE between homogeneous Markov jump processes and discuss some open problems and possible future directions.
\end{abstract}

\section{Introduction}
\label{intro}
The quantification of causal relationships between time series is a fundamental problem in fields including, for example, neuroscience (\cite{battaglia2012dynamic,ito2011extending,neurobook,wollstadt2014efficient}), social networking (\cite{he2013identifying,ver2012information}), finance (\cite{debowski,kwon2008information,sandoval,sensoy2014effective}), and machine learning (\cite{Herzog2017TransferEF,obst2010improving}).  Among the various means of measuring such relationships, information theoretical approaches are a rapidly developing area in concert with other paradigms such as Pearl semantics and Granger causality.  One such approach is to make use of the notion of transfer entropy, which we abbreviate throughout as ``TE''.  Broadly speaking, transfer entropy is a functional which measures the information transfer between two stochastic processes.  Schreiber's definition of transfer entropy \cite{schreiber} characterizes information transfer as an informational divergence between conditional probability mass functions.  The original definition is native to discrete space processes indexed over a countable set, often the natural numbers.  One can generalize Schreiber's definition to handle the case when the random variables comprising the process have state space $\mathbb{R}$ via the Radon-Nikodym Theorem as demonstrated in \cite{article}.  While this formalism is applicable to some practical scenarios, it suffers from a serious deficiency: it is only applicable to processes defined over discrete time.  

A treatment of TE for processes that are either indexed over an uncountable set or do not have $\mathbb{R}$ as the state space of their constituent variables has been lacking in the literature.  A common workaround to this shortcoming is the approach of time-binning which has been widely used as a means to capture intuitively the notion of information transfer between processes (\cite{hlavavckova2007causality,lee2012transfer,liang2013liang}).  These approaches, while sometimes effective and practicable, do not provide a native definition of TE in continuous time; that is, TE between processes indexed over an uncountable set. Recently, Spinney, Prokopenko, and Lizier (\cite{spl}) set out a framework to remedy this gap.  We formalize this approach and explore the consequences by providing a definition of TE for discrete time processes comprised of random variables with a Polish state space and extend this definition to continuous time processes via projective limits, random measures, and the Radon-Nikodym Theorem.  In Section \ref{a recasting}, we provide our main result, Theorem \ref{main}, which characterizes when our continuous time definition of TE can be obtained as a limit of discrete time TE and apply it to a time-lagged Poisson point process in Section \ref{Application: Lagged Poisson point process}.

In some applications, the instantaneous transfer entropy is of particular interest.  Using our methodology, we define the transfer entropy rate (TE rate) as the right derivative with respect to time of the expected pathwise transfer entropy (EPT) functional defined in Section \ref{PT_EPT_sec} and demonstrate some of its basic properties, including a precise version of a result stated without proof in \cite{spl} regarding a particularly well-behaved class of stationary processes.  In Section \ref{cadlag_sec}, we consider time-homogeneous Markov jump processes and provide an analytic form of the EPT via a Girsanov formula.  We finish with several open questions and directions for future work, as well as an Appendix which provides some relevant calculations regarding TE in the context of Wiener processes.

\section{Discrete Time Transfer Entropy over a Polish Space}
Suppose  $X:=\{X_{n}\}_{n\geq 1}$ and $Y:=\{Y_{n}\}_{n \geq 1}$ are stochastic processes adapted to the filtered probability space  $\left(\Omega, \mathcal{F}, \{\mathcal{F}_{n}\}_{n \geq 1},\mathbb{P}\right)$.  Suppose further that for each $n \geq 1$,   $X_{n}$  and $Y_{n}$ are random variables taking values in a Polish state space $\Sigma$, i.e., a completely metrizable, separable space; and let $\mathcal{X}$ be a $\sigma$-algebra of subsets of $\Sigma$.  Denote by $\mathbb{P}_{n}$ the probability distribution of the random variable $X_{n}$ (by which sometimes we will mean a conditional probability distribution). For integers $k,l,n\geq 1$, we denote the ``history vectors'' of $X$ and $Y$ by
$$
\left(X_{n-k-1}^{n-1}\right) = \left(X_{n-k-1}, X_{n-k},..., X_{n-1}\right) $$ and $$\left(Y_{n-l-1}^{n-1}\right) = \left(Y_{n-l-1}, Y_{n-l},..., Y_{n-1}\right).
$$
Since $\Sigma$ is Polish, for each $k,l,n\geq 1$, there exist functions (in particular, regular conditional probability measures\footnote{The existence of regular conditional probability measures is guaranteed on Polish spaces (see Theorem 6.16 of \cite{panangaden2009labelled})}) $\mathbb{P}_{n}^{(k,l)}\left[X_{n}\middle| \left(X_{n-k-1}^{n-1}\right), \left(Y_{n-l-1}^{n-1}\right)\right]$ and $\mathbb{P}_{n}^{(k)}\left[X_{n}\middle| \left(X_{n-k-1}^{n-1}\right)\right]$ mapping $\mathcal{F}_{n}\times \Omega$ to $[0,1]$ with the following properties:
\begin{itemize}
\item[1.]For each $\omega\in\Omega$, both \begin{equation}\label{fix_before_measure_X}
\mathbb{P}_{n}^{(k)}\left[X_{n}\middle| \left(X_{n-k-1}^{n-1}\right)\right](\cdot, \omega)
\end{equation}
and
\begin{equation}
\mathbb{P}_{n}^{(k,l)}\left[X_{n} \middle|  \left(X_{n-k-1}^{n-1}\right), \left(Y_{n-l-1}^{n-1}\right)\right](\cdot, \omega)
\end{equation} are measures on $ \left(\Sigma, \mathcal{X}\right)$.
\item[2.]
$\forall  A \in \mathcal{F}_{n}$  the mappings \begin{equation*}
\omega \mapsto \mathbb{P}_{n}^{(k,l)}\left[X_{n} \middle|  \left(X_{n-k-1}^{n-1}\right), \left(Y_{n-l-1}^{n-1}\right)\right](A, \omega) \end{equation*}
and
\begin{equation*}
\omega \mapsto \mathbb{P}_{n}^{(k,l)}\left[{X_{n}} \middle|  \left(X_{n-k-1}^{n-1}\right)\right](A,\omega)
\end{equation*}
are $\mathcal{F}_{n}-$ measurable random variables.

\item[3.]For all $\omega \in \Omega$ and $ A\in \mathcal{F}_{n}$ we have both
\begin{equation*}
\begin{aligned}
& \mathbb{P}_{n}^{(k,l)}\left[X_{n} \middle|  \left(X_{n-k-1}^{n-1}\right), \left(Y_{n-l-1}^{n-1}\right)\right](A, \omega) = \\
& \mathbb{P}_{n}^{(k,l)}\left[\left\{X_{n}\in A\right\} \middle |  \left\{B \in \sigma\left(\left(X_{n-k-1}^{n-1}\right), \left(Y_{n-l-1}^{n-1}\right)\right): \omega \in B\right\}\right]
\end{aligned}
\end{equation*}
and 
\begin{equation*}
\begin{aligned} &\mathbb{P}_{n}^{(k)}\left[X_{n} \middle|  \left(X_{n-k-1}^{n-1}\right)\right](A, \omega)= \\&  \mathbb{P}_{n}^{(k)}\left[\left\{X_{n}\in A\right\} \middle|  \left\{B \in \sigma\left(\left(X_{n-k-1}^{n-1}\right)\right) : \omega \in B\right\}\right].
\end{aligned}
\end{equation*}
\end{itemize}
If $\omega\in\Omega$, the conditional probabilities $\mathbb{P}_{n}^{(k,l)}\left[X_{n} \middle|  \left(X_{n-k-1}^{n-1}\right), \left(Y_{n-l-1}^{n-1}\right)\right](\cdot, \omega)$ and $\mathbb{P}_{n}^{(k)}\left[X_{n} \middle|  \left(X_{n-k-1}^{n-1}\right)\right](\cdot, \omega)$ are only defined in the case that each event $\left\{B \in \sigma\left(\left(X_{n-k-1}^{n-1}\right)\right) : \omega \in B\right\}$ and $\left\{B \in \sigma\left(\left(X_{n-k-1}^{n-1}\right), \left(Y_{n-l-1}^{n-1}\right)\right): \omega \in B\right\}$ is not a $\mathbb{P}$-null set. We will assume neither of these sets are $\mathbb{P}$-null throughout this work whenever dealing with conditional probabilities.  
\begin{Notation} For sake of convenience let
$$
\mathbb{P}_{n}^{(k)}\left[X_{n} \middle|   \left(X_{n-k-1}^{n-1}\right)\right](A, \omega) := \mathbb{P}_{n}^{(k)}\left[X_{n} \middle|   \left(X_{n-k-1}^{n-1}\right)\right] (\omega)\left( A\right)
$$ 
and 
\begin{align*}
\mathbb{P}_{n}^{(k,l)}\left[X_{n}  \middle|   \left(X_{n-k-1}^{n-1}\right), \left(Y_{n-l-1}^{n-1}\right)\right](A, \omega) := \\  \mathbb{P}_{n}^{(k,l)}\left[X_{n} \middle|   \left(X_{n-k-1}^{n-1}\right), \left(Y_{n-l-1}^{n-1}\right)\right](\omega)\left(A \right)
\end{align*}
whenever $n,k,l\geq 1$, $\omega\in\Omega$, and $A\in\mathcal{F}_{n}$. 
\end{Notation}

The following definition generalizes Schreiber's definition of TE for discrete time processes whose random variables have a Polish state space.
\begin{definition}\label{disc_TE}
Suppose $n,k,l \geq 1$ are integers.  Suppose further that $\Sigma$ is a Polish space and that 
\begin{equation}\label{abscont}
\mathbb{P}_{n}^{(k)}\left[X_{n} \middle|  \left(X_{n-k-1}^{n-1}\right), \left(Y_{n-l-1}^{n-1}\right)\right] (\omega)\ll \mathbb{P}_{n}^{(k)}\left[X_{n} \middle|  \left(X_{n-k-1}^{n-1}\right)\right] (\omega)
\end{equation}
for each $\omega\in\Omega.$ Define the {\em transfer entropy from }$Y$ {\em to }$X$ {\em at }$n$ {\em with history window lengths} $k$ {\em and }$l$, denoted $\mathbb{T}_{Y \rightarrow X}^{(k,l)}(n)$, by
\begin{small}
\begin{equation}\label{TE_def}
\begin{aligned}
\mathbb{T}_{Y \rightarrow X}^{(k,l)}(n)
= \mathbb{E}_{\mathbb{P}}\left[ KL\left(\mathbb{P}_{n}^{(k,l)}\left[X_{n} \middle|  \left(X_{n-k-1}^{n-1}\right),  \left(Y_{n-l-1}^{n-1}\right)\right]  \middle| \middle|  \mathbb{P}_{n}^{(k)}\left[X_{n} \middle|  \left(X_{n-k-1}^{n-1}\right)\right] \right)\right]
\end{aligned}
\end{equation}
\end{small}
and call $X$ the ``destination process'' and $Y$ the ``source process''.
\end{definition}
\begin{Observation}

Due to \cite{reg_cond_ent}, we have the following for each $n \geq 1$:
\begin{itemize}
\item[1.] For fixed $k,l \geq 1$, $\mathbb{T}_{Y \rightarrow X}^{(k,l)}$ is a measurable function from $\mathbb{N}$ into the extended nonnegative real line. 
\item[2.]  $KL\left(\mathbb{P}_{n}^{(k,l)}[X_{n} \middle|   (X_{n-k-1}^{n-1}),  (Y_{n-l-1}^{n-1})](\omega)  \middle| \middle|  \mathbb{P}_{n}^{(k)}[X_{n} \middle|   (X_{n-k-1}^{n-1})](\omega) \right) \geq 0, \forall \omega \in \Omega.$
\item[3.]  $\frac{d\mathbb{P}_{n}^{(k,l)}[X_{n} \mid   (X_{n-k-1}^{n-1}), (Y_{n-l-1}^{n-1})](\cdot)}{d\mathbb{P}_{n}^{(k)}[X_{n} \mid   (X_{n-k-1}^{n-1}) ](\cdot)}\left( \cdot \right)$ is $\mathcal{F} \times \mathcal{X}$-measurable as $X$ is adapted to $\mathcal{F}$.
\item[4.] For all $\omega \in \Omega$, 
$$
KL\left(\mathbb{P}_{n}^{(k,l)}[X_{n}| (X_{n-k-1}^{n-1}),  (Y_{n-l-1}^{n-1})](\omega)  \middle| \middle|  \mathbb{P}_{n}^{(k)}[X_{n}| (X_{n-k-1}^{n-1})](\omega) \right)
$$
is $\mathcal{F}-$measurable.
\end{itemize}

\end{Observation}
  
\begin{Example}\label{def_recovery}
Suppose $X$ and $Y$ are discrete processes; that is, for each integer $n \geq 1$,
both $X_{n}(\Omega) $ and $Y_{n}(\Omega)$ are countable. Then 
\begin{align*}
\label{TE_expand}
&\mathbb{T}_{Y \rightarrow X}^{(k,l)}(n) = \\ 
&\mathbb{E}_{\mathbb{P}}\left[\mathbb{E}_{\mathbb{P}_{n}^{(k,l)}\left[X_{n} \middle|  \left(X_{n-k-1}^{n-1}\right), \left(Y_{n-l-1}^{n-1}\right)\right] }     \left[\log{\frac{d\mathbb{P}_{n}^{(k,l)}\left[X_{n} \middle|   \left(X_{n-k-1}^{n-1}\right), \left(Y_{n-l-1}^{n-1}\right)\right]}{d\mathbb{P}_{n}^{(k)}\left[X_{n} \middle|   \left(X_{n-k-1}^{n-1}\right) \right]}} \right]\right]\\ 
\\&=\sum _{\substack{x_{n-k-1}^{n-1} \in X_{n-k-1}^{n-1}(\Omega)\\ y_{n-l-1}^{n-1} \in Y_{n-l-1}^{n-1}(\Omega)} } \mathbb{P}_{X_{n-k-1}^{n-1}, Y_{n-l-1}^{n-1}}\left(x_{n-k-1}^{n-1}, y_{n-l-1}^{n-1}\right)\times \\& \quad \sum_{x_{n} \in X_{n}(\Omega)}\mathbb{P}_{X_{n}\mid \left(X_{n-k-1}^{n-1}\right), \left(Y_{n-l-1}^{n-1}\right)}\left[x_{n} \middle|   \left(x_{n-k-1}^{n-1}\right), \left(y_{n-l-1}^{n-1}\right)\right]\times\\ &\qquad \log{\frac{\mathbb{P}_{X_{n}\mid \left(X_{n-k-1}^{n-1}\right), \left(Y_{n-l-1}^{n-1}\right)}\left[x_{n} \middle|  (x_{n-k-1}^{n-1}), \left(y_{n-l-1}^{n-1}\right)\right]}{\mathbb{P}_{X_{n}\mid \left(X_{n-k-1}^{n-1}\right) }\left[x_{n} \middle|  \left(x_{n-k-1}^{n-1}\right) \right]}} 
\\&= \sum _{\substack{x_{n}\in X_{n}(\Omega),\\ x_{n-k-1}^{n-1} \in x_{n-k-1}^{n-1}(\Omega),\\ y_{n-l-1}^{n-1} \in Y_{n-l-1}^{n-1}(\Omega)}} \mathbb{P}_{X_{n}, X_{n-k-1}^{n-1}, Y_{n-l-1}^{n-1}}\left(x_{n}, x_{n-k-1}^{n-1}, y_{n-l-1}^{n-1}\right)\times 
\\& \qquad \log{\frac{\mathbb{P}_{X_{n}\mid \left(X_{n-k-1}^{n-1}\right), \left(Y_{n-l-1}^{n-1}\right)}\left[x_{n} \middle|   \left(x_{n-k-1}^{n-1}\right), \left(y_{n-l-1}^{n-1}\right)\right]}{\mathbb{P}_{X_{n}\mid \left(X_{n-k-1}^{n-1}\right)}\left[x_{n} \middle|   \left(x_{n-k-1}^{n-1}\right) \right]}}
\end{align*}
where the RN-derivatives have become quotients of {\em probability mass functions} since the processes is composed of discrete random variables.  The above demonstrates that Schreiber's initial definition of transfer entropy is indeed a special case of our more general definition of TE.  Furthermore, 
if $(\Sigma, \mathcal{X}) = (\mathbb{R}, \mathcal{B}(\mathbb{R}))$ and the joint probability measure $\mathbb{P}_{X_{n}, \left(X_{n-k-1}^{n-1}\right),\left( Y_{n-l-1}^{n-1}\right)}$ is absolutely continuous with respect to Lebesgue measure on $\mathbb{R}^{(1+k+l)}$, then there exist RN-derivatives ({\em probability densities})
\begin{equation}
    p_{X_{n},\left(X_{n-k-1}^{n-1}\right), \left(Y_{n-l-1}^{n-1}\right)}, p_{X_{n}\mid \left(X_{n-k-1}^{n-1}\right), \left(Y_{n-l-1}^{n-1}\right)} \text{ and }p_{X_{n}\mid \left(X_{n-k-1}^{n-1}\right)}
\end{equation}
which can replace the probability mass functions in Schreiber's definition.  In regards to our definition in this setting, $\mathbb{R}$ is indeed Polish, thus assuming (\ref{abscont}) our definition yields
\begin{equation*}
\begin{aligned}
&\mathbb{T}_{Y \rightarrow X}^{(k,l)}(n) \\&= \int_{\mathbb{R}^{(1+k+l)}}p_{X_{n},\left(X_{n-k-1}^{n-1}\right), \left(Y_{n-l-1}^{n-1}\right)}\left(x_{n},\left(x_{n-k-1}^{n-1}\right),\left( y_{n-l-1}^{n-1}\right)\right)\times\\ &
\qquad \log\left(\frac{p_{X_{n}\mid \left(X_{n-k-1}^{n-1}\right), \left(Y_{n-l-1}^{n-1}\right)}\left(x_{n}\mid\left(x_{n-k-1}^{n-1}\right), \left(y_{n-l-1}^{n-1}\right)\right)}{p_{X_{n}\mid \left(X_{n-k-1}^{n-1}\right)}\left(x_{n}\mid \left(x_{n-k-1}^{n-1}\right) \right)}\right) d\mu_{(1+k+l)}
\end{aligned}
\end{equation*}
where $\mu_{(1+k+l)}$ denotes Lebesgue measure on $\mathbb{R}^{(1+k+l)}.$  This expression is exactly that for TE in this special case (see \cite{article}); thus, our definition recovers the correct expression for TE in the case that $(\Sigma, \mathcal{X}) = (\mathbb{R}, \mathcal{B}(\mathbb{R}))$ as well. 
\end{Example}

Note that Definition \ref{disc_TE} differs somewhat from the definition of TE in \cite{spl}, in that we employ two expectations.  The idea of using two expectations to represent some of the more common conditional versions of information-theoretical functionals has appeared in other works (see Section 3 of \cite{def_int_2} and (14) in \cite{def_int_1}).

\section{Construction of path measures}
\label{CPM}
We now turn our attention to the main purpose of this work, namely, developing TE in continuous time.  We restrict our attention to the case when the uncountable indexing set is an interval.  Let $\mathbb{T} \subset \mathbb{R}_{\geq 0}$ be a closed and bounded interval whose elements we refer to as {\em times}.  Analogous to the setup for discrete time TE, we suppose $X:=\{X_{t}\}_{t\in \mathbb{T}}$ and $Y:=\{Y_{t}\}_{t\in \mathbb{T}}$ are stochastic processes adapted to the filtered probability space $\left(\Omega, \mathcal{F}, \{\mathcal{F}_{t}\}_{t \in \mathbb{T}},\mathbb{P}\right)$ such that for each $t \in \mathbb{T}$, $X_{t}$  and $Y_{t}$ are random variables taking values in the measurable state space $\left(\Sigma, \mathcal{X}\right)$ where $\Sigma$ is a Polish space and $\mathcal{X}$ is a $\sigma-$algebra of subsets of $\Sigma$.  
In this section we begin our construction of continuous time TE by introducing conditional measures on the space of sample paths of $X$. These measures will act as the continuous time analogues of the random conditional probabilities  
$$
\mathbb{P}_{n}^{(k,l)}\left[X_{n}\Big| \left(X_{n-k-1}^{n-1}\right), \left(Y_{n-l-1}^{n-1}\right)\right]
$$
and 
$$
\mathbb{P}_{n}^{(k)}\left[X_{n}\Big| \left(X_{n-k-1}^{n-1}\right)\right]
$$
in Definition \ref{disc_TE}.  The following seminal result in \cite{rao} will be crucial to the formulation of these measures.  
\begin{theorem}\label{rao_theorem}
Let $\mathbb{A}$ be any index set and $D$ the set of all its finite subsets directed by inclusion.  Let $\left(\Sigma_{t}, \mathcal{X}_{t}\right)_{t\in\mathbb{A}}$ be a family of measurable spaces where $\Sigma_{t}$ is a
topological space and $\mathcal{X}_{t}$ is a $\sigma$-field containing all the compact subsets of $\Sigma_{t}$. Suppose, for $\alpha\in D$, $\Sigma_{\alpha} = \times_{t\in\alpha}\Sigma_{t}, \mathcal{X}_{\alpha} = \bigotimes_{t\in\alpha}\mathcal{X}_{t},$ and $\mathbb{P}_{\alpha}:\mathcal{X}_{\alpha}\mapsto [0,1]$ so that $\left(\Sigma_{\alpha}, \mathcal{X}_{\alpha}, \mathbb{P}_{\alpha} \right)$ is a probability space.  If for each $\alpha \in D$, $\mathbb{P}_{\alpha}$ is inner regular relative to the compact subsets of $\mathcal{X}_{\alpha}$, i.e., for any $A\in \mathcal{X}_{\alpha}$, $\mathbb{P}_{\alpha} = \sup\left\{ \mathbb{P}_{\alpha}(C) :   C \text{ is a compact subset of } A\right\},$ and $\pi_{\alpha\beta}: \Sigma_{\beta}\mapsto \Sigma_{\alpha}$ $(\beta \geq \alpha)$, $\pi_{\alpha} = \pi_{\alpha\mathbb{A}}: \times_{t\in\mathbb{A}}\Sigma_{t} \mapsto  \Sigma_{\alpha}$ for $\alpha,\beta\in D$ are coordinate projections, then there exists a unique probability measure $\mathbb{P}_{\mathbb{A}}$ on the space $\left(\times_{t\in\mathbb{A}}\Sigma_{t},\bigotimes_{t\in\mathbb{A}}\mathcal{X}_{t} \right)$ such that $\forall \alpha \in D$,  \begin{equation}\label{rao_conc}\mathbb{P}_{\alpha} = \mathbb{P}_{\mathbb{A}}\circ\pi_{\alpha}^{-1},
\end{equation}if and only if $\left\{\left(\Sigma_{\alpha}, \mathcal{X}_{\alpha}, \mathbb{P}_{\alpha}, \pi_{\alpha\beta}\right)_{\beta \geq \alpha}: \alpha,\beta\in D\right\}$ is a {\em projective system} with respect to mappings $\left\{ \pi_{\alpha\beta}\right\}$; that is, \begin{itemize}
\item[(1)]$\pi_{\alpha\beta}^{-1}\left(\mathcal{X}_{\alpha}\right) \subset \mathcal{X}_{\beta} $ so that $\pi_{\alpha\beta}$ is $\left(\mathcal{X}_{\beta}, \mathcal{X}_{\alpha} \right)-$measurable.
\item[(2)] for any $\alpha \leq \beta\leq \lambda$, $\pi_{\alpha\beta}\circ \pi_{\beta\lambda} = \pi_{\alpha,\lambda}$, $\pi_{\alpha\alpha} = id_{\alpha}$ and 
\item[(3)]$\mathbb{P}_{\alpha} = \mathbb{P}_{\beta}\pi_{\alpha\beta}^{-1},$ whenever $\alpha \leq \beta.$
\end{itemize}
\end{theorem}
Due to Corollary 15.27 of \cite{infdim}, the same result holds without the inner regularity of $\mathbb{P}_{\{\cdot\}}$ whenever $\Sigma_{t}$ is a Polish space for each $t\in{\mathbb{A}}$. Furthermore, the same result holds if $D$ is the set of countably finite subsets of $\mathbb{A}$ (see Corollary 4.9.16 of 
 \cite{countable}).
 
Let $\mathbb{A} = [t_{0},T)\subset \mathbb{T}$. 
As shown in the proof of Theorem 1 (see \cite{rao}), the projective limit $\sigma - $ algebra, $\bigotimes_{t\in\mathbb{A}}\mathcal{X}_{t}$, is generated by $ \bigcup_{\alpha\in D}\pi^{-1}_{\alpha}\left(\mathcal{X}_{\alpha}\right)$; that is, 
$$
\bigotimes_{t\in\mathbb{T}}\mathcal{X}_{t} = \sigma \left(\bigcup_{\alpha\in D}\pi^{-1}_{\alpha}\left(\mathcal{X}_{\alpha}\right) \right).
$$
If $\alpha,\beta\in D$ with $\alpha < \beta,$ then due to (1) of Theorem 1 we have
\begin{equation}\label{filtration}
\pi_{\alpha}^{-1}(\mathcal{X}_{\alpha}) = \left( \pi_{\alpha\beta}\circ\pi_{\beta}\right)^{-1}(\mathcal{X}_{\alpha}) \subset \pi_{\beta}^{-1}(\mathcal{X}_{\beta}).
\end{equation}
Consequently, $\left( \pi_{\alpha}^{-1}(\mathcal{X}_{\alpha})\right)_{\alpha \in D}$ is a filtration ordered by set inclusion which generates $\bigotimes_{t\in\mathbb{A}}\mathcal{X}_{t}$ and from (\ref{rao_conc}) we have
\begin{equation}\label{arb_filter}
\mathbb{P}_{\mathbb{A}}\mid_{\pi_{\alpha}^{-1}(\mathcal{X}_{\alpha})} = \mathbb{P}_{\alpha}\circ \pi_{\alpha}.
\end{equation}
In our case, we assume that $\Sigma_{t} = \Sigma$ and $\mathcal{X}_{t} = \mathcal{X}$ for all $ t\in\mathbb{T}.$

Now let $s,r > 0$ be such that $\left( t_{0} - \max{(s,r)}, T\right) \subset \mathbb{T}$.  The numbers $s$ and $r$ are in place to act as the analogues of the positive integers $k$ and $l$ in Definition \ref{disc_TE}.  For each $\Delta t>0$ define the {\em comb set} $D_{\Delta t}\subset \mathbb{T}$ by 
\begin{equation*}
    \begin{aligned}
   D_{\Delta t} & = 
   \Big\{\left\lfloor \frac{t_{0}}{\Delta t}\right\rfloor\Delta t - \left(\left\lfloor \frac{W}{\Delta t}\right\rfloor -1\right)\Delta t,\dots, \left\lfloor\frac{t_{0}}{\Delta t}\right\rfloor\Delta t, \left\lfloor\frac{ t_{0}}{\Delta t}\right\rfloor\Delta t + \Delta t,  \dots\\
   & \qquad \dots, \left\lfloor\frac{T}{\Delta t}\right\rfloor\Delta t -2\Delta t,\left\lfloor\frac{T}{\Delta t}\right\rfloor\Delta t -\Delta t ,  \left\lfloor\frac{T}{\Delta t}\right\rfloor\Delta t \Big\}
    \end{aligned}
\end{equation*}
where $W = \max{(s,r)}$.
\begin{Notation} Henceforth, we will let $\tau = \left\lfloor\frac{T}{\Delta t} \right\rfloor- \left\lfloor\frac{t_{0}}{\Delta t}\right\rfloor$ and $\left\langle T, i, \Delta t\right\rangle = \left\lfloor\frac{T}{\Delta t} \right\rfloor \Delta t - i\Delta t$ for $\Delta t>0, i = 0,1, \dots, \tau - 1$.
\end{Notation}
Given $\Delta t>0$ we can use the comb set $D_{\Delta t}$ to construct two probability measures on the measurable space $\left(\Sigma^{\tau}, \bigotimes_{i=0}^{\tau -1}\mathcal{X}\right)$. Specifically, for $\Delta t>0$ let $A_{m}^{\Delta t, X} = \left\{ X_{m}\in B_{m}\right\}$, $A_{m}^{\Delta t, Y} = \left\{ Y_{m}\in B_{m}\right\}$, $X_{m,k}^{\Delta t} = \sigma\left(\left(X_{\left\langle T,m+k+1,\Delta t \right\rangle}^{\left\langle T,m+1,\Delta t \right\rangle}\right)\right)$, and $Y_{m, k,l}^{\Delta t} = \sigma\left(\left(Y_{\left\langle T,m+l+1,\Delta t \right\rangle }^{\left\langle T,m+1,\Delta t \right\rangle}\right)\right)$ 
 for each $m = 0,1,\cdots, \tau-1$.
 Then 
\begin{equation}\label{eq.1}
\prod_{i=0}^{\tau-1} \mathbb{P}_{ \left\langle T, i, \Delta t\right\rangle}\left (A_{ \left\langle T, i , \Delta t\right\rangle  }^{\Delta t, X}\middle| \alpha_{X}^{i,\Delta t}\right) 
=\prod_{i=0}^{\tau-1}\left(\mathbb{P}_{\left\langle T, i, \Delta t\right\rangle}\left( X_{\left\langle T, i, \Delta t\right\rangle}  \middle|  X_{i,k}^{\Delta t}\right) \right)(\omega)\left( B_{\left\langle T, i, \Delta t\right\rangle}\right)
\end{equation}
for some $\omega \in \Omega$ where $k = \left\lfloor\frac{s}{\Delta t}\right\rfloor$ 
and 
$\alpha_{X}^{i,\Delta t} = \bigcap _{ j = \left\lfloor\frac{T}{\Delta t}\right\rfloor-( i+\left\lfloor \frac{s}{\Delta t}\right\rfloor +1) }^{\left\lfloor\frac{T}{\Delta t}\right\rfloor -(i+1) }A_{j\Delta t}^{\Delta t, X}$.  Similarly,
\begin{equation}\label{eq.2}\begin{aligned}
 & \prod_{i=0}^{\tau-1} \mathbb{P}_{ \left\langle T, i, \Delta t\right\rangle}\left (A_{\left\langle T, i, \Delta t\right\rangle   }^{\Delta t, X} \middle| \left( \alpha_{X}^{i,\Delta t}\right)\bigcap \left( \alpha_{Y}^{i,\Delta t}\right)\right)
 \\& =  \prod_{i=0}^{\tau-1}\left(\mathbb{P}_{\left\langle T, i, \Delta t\right\rangle}\left( X_{\left\langle T, i, \Delta t\right\rangle} \middle|  X_{i,k}^{\Delta t}, Y_{i,k,l}^{\Delta t}\right) \right)(\omega)\left( B_{\left\langle T, i, \Delta t\right\rangle}\right),
\end{aligned}\end{equation}
for some $\omega \in \Omega$ where $l = \left\lfloor\frac{r}{\Delta t}\right\rfloor$ and $\alpha_{Y}^{i,\Delta t} = \bigcap_{ j = \left\lfloor\frac{T}{\Delta t}\right\rfloor-( i+\left\lfloor \frac{r}{\Delta t}\right\rfloor +1) }^{ \left\lfloor\frac{T}{\Delta t}\right\rfloor -(i+1) }A_{j\Delta t}^{\Delta t, Y}$. 
Given $\omega \in \Omega$ and $ \Delta t>0$ define the measures $\mathbb{P}_{X\mid \overleftarrow{X},i,\Delta t}^{(\omega),(k)}$ and  $\mathbb{P}_{X\mid \overleftarrow{X},\overleftarrow{Y},i,\Delta t}^{(\omega),(k,l)}$ on the space $\left(\Sigma, \mathcal{X} \right)$ for each $i = 0,1,\cdots, \tau-1$ by \begin{equation}\label{fdd_xy}
\begin{aligned}
\mathbb{P}_{X\mid \overleftarrow{X},i,\Delta t}^{(\omega),(k)}\left( B_{\left\langle T, i, \Delta t\right\rangle}\right)  = \left(\mathbb{P}_{\left\langle T, i, \Delta t\right\rangle}\left( X_{\left\langle T, i, \Delta t\right\rangle}   \big| X_{i,k}^{\Delta t}\right) \right)(\omega)\left( B_{\left\langle T, i, \Delta t\right\rangle}\right)
\end{aligned}
\end{equation} and
\begin{equation}\label{fdd_x}
\begin{aligned}
&\mathbb{P}_{X\mid \overleftarrow{X},\overleftarrow{Y},i,\Delta t}^{(\omega),(k,l)} \left( B_{\left\langle T, i, \Delta t\right\rangle}\right)= \left(\mathbb{P}_{\left\langle T, i, \Delta t\right\rangle}\left( X_{\left\langle T, i, \Delta t\right\rangle}   \big| X_{i,k}^{\Delta t}, Y_{i,k,l}^{\Delta t}\right) \right)(\omega)\left( B_{\left\langle T, i, \Delta t\right\rangle}\right).
\end{aligned}
\end{equation}
\begin{Notation}
For $\Delta t',\Delta t>0$ we write $\Delta t' \mid \Delta t$ whenever there exists a positive integer $m$ such that $\Delta t = m\Delta t'.$
\end{Notation}
Suppose $k=\left\lfloor \frac{s}{\Delta t}\right\rfloor$ and $l=\left\lfloor \frac{r}{\Delta t}\right\rfloor$.  If for each $\omega\in\Omega$ the systems 
$$
\left\{\left(\Sigma^{\tau}, \bigotimes^{\tau}\mathcal{X},\prod_{i=0}^{\tau-1}\mathbb{P}_{X\mid \overleftarrow{X},i,\Delta t}^{(\omega),(k)},\pi_{D_{\Delta t}D_{\Delta t'}}\right)_{ \substack{0<\Delta t' < \Delta t \\ \Delta t' \mid \Delta t }} : \Delta t>0\right\}
$$
and 
$$
\left\{\left(\Sigma^{\tau}, \bigotimes^{\tau}\mathcal{X},\prod_{i=0}^{\tau-1}\mathbb{P}_{X\mid \overleftarrow{X},\overleftarrow{Y},i,\Delta t}^{(\omega),(k,l)},\pi_{D_{\Delta t}D_{\Delta t'}}\right)_{ \substack{0<\Delta t' < \Delta t \\ \Delta t' \mid \Delta t }} : \Delta t>0\right\}
$$
are projective systems with respect to coordinate projections $\left\{ \pi_{D_{\Delta t}D_{\Delta t'}}\right\}$, then as a consequence of Theorem \ref{rao_theorem}, there exist unique probability measures $$\mathbb{P}_{X}^{(s)}[X_{t_{0}}^{T}\mid X_{t_{0} - s}^{t_{0}} ](\omega)$$ and $$\mathbb{P}_{X\mid X, Y}^{(s,r)}[X_{t_{0}}^{T}\mid X_{t_{0} - s}^{t_{0}},  Y_{t_{0} - r}^{T}](\omega)$$ on the measurable space $\left(\times_{t\in[t_{0},T)}\Sigma, \bigotimes_{t\in[t_{0},T)}\mathcal{X} \right)$ such that
\begin{equation}\label{X_given_X}
\mathbb{P}_{X}^{(s)}[X_{t_{0}}^{T}\mid X_{t_{0} - s}^{t_{0}} ](\omega)\Big|_{\mathcal{F}_{{\Delta t}}^{[t_{0},T)}} = \left(\prod_{i=0}^{\tau-1}\mathbb{P}_{X\mid \overleftarrow{X},i,\Delta t}^{(\omega),(k)}\right)\circ \pi_{ D_{\Delta t} }
\end{equation}
and 
\begin{equation}\label{X_given_XY}
\mathbb{P}_{X\mid X, Y}^{(s,r)}[X_{t_{0}}^{T}\mid X_{t_{0} - s}^{t_{0}},  Y_{t_{0} - r}^{T}](\omega)\Big|_{\mathcal{F}_{{\Delta t}}^{[t_{0},T)}} = \left(\prod_{i=0}^{\tau-1}\mathbb{P}_{X\mid \overleftarrow{X},\overleftarrow{Y}, i,\Delta t}^{(\omega),(k,l)}\right)\circ \pi_{ D_{\Delta t} }
\end{equation}
where $\mathcal{F}_{\Delta t}^{[t_{0},T)} = \pi_{D_{\Delta t}}^{-1}(\mathcal{X}_{D_{\Delta t}})$. 

\begin{Notation} Let $\Omega_{X}^{[t_{0},T)}$ denote the set of sample paths of $X$.
\end{Notation}

\section{Pathwise transfer entropy and expected pathwise transfer entropy}
\label{PT_EPT_sec}
The purpose of this section is to use the  measures $ \mathbb{P}_{X}^{(s)}[X_{t_{0}}^{T}\mid X_{t_{0} - s}^{t_{0}} ](\cdot)$ and $\mathbb{P}_{X \mid   X, Y}^{(s,r)}\left[X_{t_{0}}^{T}\mid X_{t_{0} - s}^{t_{0}},  Y_{t_{0} - r}^{T}\right](\cdot)$ to define transfer entropy over an interval of the form $[t_{0}, T)\subset \mathbb{T}$ with history window lengths $r,s>0$.
\begin{definition}
Suppose $\mathbb{T}\subset \mathbb{R}_{\geq 0}$ is a closed and bounded interval, $[t_{0},T) \subset \mathbb{T}$; $r,s > 0$; and for each $\omega \in \Omega$ the measures $\mathbb{P}_{X\mid X, Y}^{(s,r)}[X_{t_0}^{T}\mid X_{t_0 - s}^{T},  Y_{t_{0} - r}^{T}](\omega)$ and $\mathbb{P}_{X}^{(s)}[X_{t_0}^{T}\mid X_{t_0 - s}^{T} ](\omega)$ exist.  If $\left( t_{0} - \max\left(s,r\right), T\right) \subset \mathbb{T}$, then for any sample path $x_{t_{0}}^{T} \in \Omega_{X}^{[t_{0},T)},$ define the {\em pathwise transfer entropy from} $Y$ {\em to} $X$ {\em on} $[t_{0},T)$ {\em at} $x_{t_{0}}^{T}$ {\em with history window lengths} $r$ {\em and} $s$, denoted ${\mathcal{P}\mathcal{T}}_{Y \rightarrow X}^{(s,r)}\mid_{t_0}^{T}(\omega,x_{t_{0}}^{T} )$, by 
\begin{equation}\label{PT}
{\mathcal{P}\mathcal{T}}_{Y \rightarrow X}^{(s,r)}\mid_{t_0}^{T}(\omega, x_{t_{0}}^{T}) = \log{\frac{d\mathbb{P}_{X\mid X, Y}^{(s,r)}[X_{t_0}^{T}\mid X_{t_{0} - s}^{t_{0}},  Y_{t_{0} - r}^{T}](\omega)}{d\mathbb{P}_{X}^{(s)}[X_{t_0}^{T}\mid X_{t_{0} - s}^{t_{0}} ](\omega)}}\left( x_{t_{0}}^{T}\right)
\end{equation}
if $\mathbb{P}_{X\mid X, Y}^{(s,r)}[X_{t_0}^{T}\mid X_{t_0 - s}^{t_{0}},  Y_{t_{0} - r}^{T}](\omega) \ll  \mathbb{P}_{X}^{(s)}[X_{t_0}^{T}\mid X_{t_0 - s}^{t_{0}} ](\omega)$ for all $\omega\in\Omega$ and $\infty$ otherwise.
\end{definition}
\begin{Observation}
 For each $\omega \in \Omega$, ${\mathcal{P}\mathcal{T}}_{Y \rightarrow X}^{(s,r)}\mid_{t_0}^{T}(\omega, \cdot)$ maps $\Omega_{X}^{[t_{0},T)}$ into the extended real line $\mathbb{R}\cup\{\infty\}$ and ${\mathcal{P}\mathcal{T}}_{Y \rightarrow X}^{(s,r)}\mid_{t_0}^{T}(\omega, \cdot)$ is unique  $ \mathbb{P}_{X}^{(s)}[X_{t_0}^{T}\mid X_{t_0 - s}^{t_{0}} ](\omega)$-\textrm{a.s.} due to the Radon-Nikodym Theorem.
\end{Observation}

The following is our definition of transfer entropy over an interval of the form $[t_{0},T)$\footnote{ One could, in principle, construct a similar definition in the case that the interval is of the form $[t_{0},T]$, via following the procedure outlined in Section \ref{CPM} with comb sets of the form $\widetilde{D_{\Delta t}} :=\left\{T, T-\Delta t, T-2\Delta t, \dots, T - \left\lfloor\frac{\max{(s,r)}}{\Delta t}\right\rfloor \Delta t \right\}$ rather than $D_{\Delta t}$.}.
\begin{definition}\label{EPT_def}
Suppose $\mathbb{T}\subset \mathbb{R}_{\geq 0}$ is a closed and bounded interval, $[t_{0},T) \subset \mathbb{T}$; $r,s > 0$; and for each $\omega \in \Omega$ the measures $\mathbb{P}_{X\mid X, Y}^{(s,r)}[X_{t_0}^{T}\mid X_{t_0 - s}^{T},  Y_{t_{0} - r}^{T}](\omega)$ and $\mathbb{P}_{X}^{(s)}[X_{t_0}^{T}\mid X_{t_0 - s}^{T} ](\omega)$ exist.  If $\left( t_{0} - \max\left(s,r\right), T\right) \subset \mathbb{T}$, the {\em expected pathwise transfer entropy (EPT) from} $Y$ {\em to} $X$ {\em on} $[t_{0},T)$ {\em with history window lengths} $r$ {\em and} $s$, denoted ${\mathcal{E}\mathcal{P}\mathcal{T}}_{Y \rightarrow X}^{(s,r)}\mid_{t_0}^{T}$, is defined by 
\begin{equation}\label{EPT}
\mathcal{E}\mathcal{P}\mathcal{T}_{Y \rightarrow X}^{(s,r)}\mid_{t_0}^{T}= 
\mathbb{E}_{\mathbb{P}}\left[
\mathbb{E}_{\mathbb{P}_{X\mid X, Y}^{(s,r)}}\left[\log{\frac{d\mathbb{P}_{X\mid X, Y}^{(s,r)}\left[X_{t_0}^{T}\middle| X_{t_{0} - s}^{t_{0}},  Y_{t_{0} - r}^{T}\right]}{d\mathbb{P}_{X}^{(s)}\left[X_{t_0}^{T}\middle| X_{t_{0} - s}^{t_{0}} \right]}}\right]\right]
\end{equation}
if $\mathbb{P}_{X\mid X, Y}^{(s,r)}\left[X_{t_{0}}^{T}\mid X_{t_{0} - s}^{t_{0}},  Y_{t_{0} - r}^{T}\right](\omega) \ll  \mathbb{P}_{X}^{(s)}\left[X_{t_0}^{T}\mid X_{t_{0} - s}^{t_{0}} \right](\omega)$ for all $\omega\in \Omega$
and $\infty$ otherwise.
\end{definition}

For the sake of clarity we emphasize that the expectation in (\ref{EPT}) is understood as the integral
\begin{equation}\label{explain}
\mathbb{E}_{\mathbb{P}}\left[
\mathbb{E}_{\mathbb{P}_{X\mid X, Y}^{(s,r)}}\log{\frac{d\mathbb{P}_{X\mid X, Y}^{(s,r)}\left[X_{t_0}^{T}\middle| X_{t_{0} - s}^{t_{0}},  Y_{t_{0} - r}^{T}\right]}{d\mathbb{P}_{X}^{(s)}\left[X_{t_0}^{T}\middle| X_{t_{0} - s}^{t_{0}} \right]}}\right] =  \int_{\Omega} KL(\omega) \; d\mathbb{P}(\omega)
\end{equation}
where 
\begin{equation*}
    \begin{aligned}
KL(\omega)= 
& \int_{\Omega_{X}^{[t_{0},T)}} \log\left[ 
\frac{d\mathbb{P}_{X\mid X, Y}^{(s,r)}[X_{t_0}^{T}\mid X_{t_o - s}^{t_{0}},  Y_{t_{0} - r}^{T}](\omega)}{d\mathbb{P}_{X}^{(s)}[X_{t_0}^{T}\mid X_{t_o - s}^{t_{0}} ](\omega)}\left(x_{t_{0}}^{T}\right)\right] \\& \qquad d\mathbb{P}_{X\mid X, Y}^{(s,r)}[X_{t_0}^{T}\mid X_{t_o - s}^{t_{0}},  Y_{t_{0} - r}^{T}](\omega)
    \end{aligned}
\end{equation*}
and note that this is similar to the expression in (\ref{TE_def}) for discrete time TE in that it is an expectation of a KL-divergence among conditional measures induced by the dynamics of $X$ and $Y$.   

\section{Obtaining continuous time TE as a limit of discrete time TE}
\label{a recasting}

We now pursue conditions under which the EPT can be represented as a limit of discrete time TE.  We first prove two lemmas that will be used in the proof of our main theorem; then we define a type of consistency between processes that makes the expressions in the main result meaningful; then we provide our main result, Theorem \ref{main}, and conclude with some of its consequences.

\begin{lemma}\label{fixedR-N}
Suppose $N \geq 1$ and $\{\mu_{i}\}_{i \geq 1}$ and $\{\nu_{i}\}_{i \geq 1}$ are finite measures on the measurable space $\left(\mathcal{X}, \Sigma\right)$ with $\mu_{i}\ll \nu_{i}$ for $i=1,...,N.$   Let $\mu = \prod_{i=1}^{N}\mu_{i}$ and $\nu = \prod_{i=1}^{N}\nu_{i}$ be product measures on the space $\left(\mathcal{X}^{N}, \otimes^{N}\Sigma\right)$.  Then $\mu \ll \nu$ and
 $$\prod_{i=1}^{N}\frac{d\mu_{i}}{d\nu_{i}}\left( \pi_{i}(x_{1},x_{2},...,x_{N})\right) = \frac{d\mu}{d\nu}\left( x_{1},x_{2},...,x_{N} \right),\, \nu-a.e.$$
where $x_{i}\in \mathcal{X}$ for $i \in [N].$

\end{lemma}
\begin{proof}
Clearly $\mu \ll \nu$ since $\forall A \in \otimes ^{N}\Sigma$ we have
\begin{equation*}
\begin{aligned}
&\nu(A)=0
\\&\implies \exists  j \leq N, \text{ such that }\nu_{j}(\pi_{j}(A))=0
\\&\implies \mu_{j}\left( \pi_{j}(A)\right)=0
\\&\implies \mu(A)=0.
\end{aligned}
\end{equation*}
Fix $E \in \otimes^{N}\Sigma$ and for $i = 1,2,\ldots,N$ let $$E_{x_{1},x_{2},...,x_{i}} = \{ (x_{i+1},x_{i+2},...,x_{N}\in \mathcal{X}^{N-i}): (x_{1},x_{2},...,x_{i},x_{i+1},...,x_{N})\in E\}$$
where $x_{i}\in \mathcal{X}, \forall i \in [N].$  Then from the Radon-Nikodym chain rule we obtain
\begin{align*}
\mu(A)&= \int_{\mathcal{X}}\ldots\int_{\mathcal{X}}\left(\chi_{E_{x_{1},x_{2},...,x_{N-1}}}(x_{N})\right)d\mu_{N}\left(x_{N}\right)...d\mu_{1}(x_{1})
\\& = \int_{\mathcal{X}}\ldots\int_{\mathcal{X}}\left(\chi_{E_{x_{1},x_{2},...,x_{N-2}}}\left(x_{N},x_{N-1}\right)\right)\prod_{i=1}^{N}\frac{d\mu_{i}}{d\nu_{i}}(x_{i})\prod_{i=1}^{N}d\nu_{i}(x_{i})
\\& \vdots\\
& = \int_{\mathcal{X}^{N}}\chi_{E}\left( x_{N},x_{N-1},\ldots,x_{2},x_{1} \right)\prod_{i=1}^{N}\frac{d\mu_{i}}{d\nu_{i}}(x_{i})\prod_{i=1}^{N}d\nu_{i}(x_{i})
\\& = \int_{E}\prod_{i=1}^{N}\frac{d\mu_{i}}{d\nu_{i}}(x_{i})\prod_{i=1}^{N}d\nu_{i}(x_{i}).
\\& =  \int_{E}\prod_{i=1}^{N}\frac{d\mu_{i}}{d\nu_{i}}(x_{i})d\nu\left( x_{1},\ldots,x_{N}\right).
\end{align*}
By the uniqueness of the RN-derivative we have
\begin{equation*}
\begin{aligned} \frac{d\mu}{d\nu}\left(x_{1},x_{2},...,x_{N} \right) &= \prod_{i=1}^{N}\frac{d\mu_{i}}{d\nu_{i}}\left( x_{i}\right) 
\\& = \prod_{i=1}^{N}\frac{d\mu_{i}}{d\nu_{i}}\left( \pi_{i}(x_{1},x_{2},...,x_{N})\right), \, \nu-\textrm{a.e.}
\end{aligned}
\end{equation*}
which completes the proof.
\end{proof}
The following lemma establishes convergence of KL-divergences in a manner which will be useful in the proof of our main result.

\begin{lemma}\label{KL_conv}
Suppose $(\Omega, \mathcal{F})$ is a measurable space.  Furthermore, suppose that $\left( \mathcal{F}_{\Delta t}\right)_{\Delta t >0}$ is a sequence of decreasing sub-$\sigma$-algebras of $\mathcal{F}$ such that $\mathcal{F} = \bigcap_{\Delta t>0} \mathcal{F}_{\Delta t}$   and that $P$ and $M$ are probability measures on $\left(\Omega, \mathcal{F} \right)$ with $P \ll M$. Let $P_{\Delta t} = P\mid_{\mathcal{F}_{\Delta t}}$ and $M_{\Delta t} = M\mid_{\mathcal{F}_{\Delta t}}$ for each $\Delta t>0$. If $\mathbb{E}_{P}\left[\log{\frac{dP}{dM}}\right]<\infty $, then  
 \begin{equation}\label{conc_lem}
 \mathbb{E}_{P_{\Delta t}}\left[\log{\frac{dP_{\Delta t}}{dM_{\Delta t}}}\right] \rightarrow \mathbb{E}_{P}\left[\log{\frac{dP}{dM}}\right] 
 \end{equation}
 as $\Delta t \downarrow 0$.
\end{lemma}
\begin{proof}
 Since probability measures are $\sigma-$finite, all RN-derivatives in (\ref{conc_lem}) exist. Suppose $\Delta t >0$.  Observe that for all $A \in \mathcal{F}_{\Delta t}$ we have that \begin{equation*}\label{mart}
\mathbb{E}_{M}\left[ \chi_{A}\frac{dP_{\Delta t}}{dM_{\Delta t}}\right]  = \mathbb{E}_{M}\left[ \chi_{A}\frac{dP}{dM}\right]
\end{equation*}
implying that
\begin{equation}\label{mart_1}
\mathbb{E}_{M}\left[ \frac{dP}{dM}\middle| \mathcal{F}_{\Delta t}\right] = \frac{dP_{\Delta t}}{dM_{\Delta t}}, M-\textrm{a.s.}
\end{equation}
from the definition of conditional expectation.
Define $\zeta_{\Delta t} = \frac{dP_{\Delta t}}{dM_{\Delta t}}$ for each $\Delta t>0.$  From (\ref{mart_1}), we get that $\left\{\zeta_{\Delta t}\right\}_{\Delta t>0}$ is a uniformly integrable backward martingale since $\zeta_{\Delta t}$ is clearly $M-$integrable for any $\Delta t>0$ by the Radon-Nikodym Theorem and \begin{equation*}
    \begin{aligned}
     \mathbb{E}_{M}\left[\zeta_{\Delta t}  \middle|   \mathcal{F}_{\Delta t'} \right]
     & = \mathbb{E}_{M}\left[\mathbb{E}_{M}\left[ \frac{dP}{dM} \middle|  \mathcal{F}_{\Delta t}\right]  \middle|   \mathcal{F}_{\Delta t'} \right]\\
     & = \mathbb{E}_{M}\left[\frac{dP}{dM}  \middle|   \mathcal{F}_{\Delta t'} \right] = \zeta_{\Delta t'}
    \end{aligned}
\end{equation*}
whenever $\Delta t'>\Delta t$ due to the tower property of conditional expectation.

We claim that 
\begin{equation}\label{rad_lim}
 \lim_{\Delta t \downarrow 0}\zeta_{\Delta t} = \frac{dP}{dM},  M-\textrm{a.s.}
\end{equation}
To see this, note first that the limit exists a.s and in $L_{1}$ due to Theorem 6.1 of \cite{Durrett}, i.e., there exists some nonnegative $\zeta \in L_{1}\left(\Omega, \mathcal{F}, M\right)$ such that
\begin{equation*}\label{L_1}
\mathbb{E}_{M}\left[\big|\zeta_{\Delta t} - \zeta\big|\right] \rightarrow 0 
\end{equation*}
as $\Delta t \downarrow 0$.
 Fix $\Delta t >0$ and suppose $A\in\mathcal{F}_{\Delta t}$. Then for all $0<\Delta t' <\Delta t$ we have that $A\in \mathcal{F}_{{\Delta t'}}$ since  $\left( \mathcal{F}_{\Delta t}\right)_{\Delta t >0}$ is a decreasing collection of $\sigma-$algebras. As a consequence of the Radon-Nikodym Theorem,  $P(A) =\mathbb{E}_{M}\left[\chi_{A} \zeta_{\Delta t'} \right]$, implying that $\mathbb{E}_{M}\left[\chi_{A} \zeta_{\Delta t'} \right]$ is  constant for $0<\Delta t' <\Delta t$. Consequently,
\begin{equation*}\label{annoy}
P(A) = \mathbb{E}_{M}\left[\chi_{A} \zeta_{\Delta t'} \right] = \mathbb{E}_{M}\left[\chi_{A} \zeta \right].
\end{equation*}
Furthermore, since $\mathcal{F} = \bigcap_{\Delta t>0}\mathcal{F}_{\Delta t}$ we must have that 
$P(A) = \mathbb{E}_{M}\left[\chi_{A} \zeta \right]$ for all $A \in \mathcal{F}$, proving (\ref{rad_lim}).
Since $(0,\infty)\ni x\mapsto x\log{x}$ is convex and $\forall \Delta t >0$, \begin{equation}
\label{thing_1}
\mathbb{E}_{P}\left[  \log{\zeta_{\Delta t}} \right]  =  \mathbb{E}_{M_{\Delta t}}\left[\zeta_{\Delta t}\log{\zeta_{\Delta t}}   \right] = \mathbb{E}_{M}\left[\frac{dP_{\Delta t}}{dM_{\Delta t}}\log{\frac{dP_{\Delta t}}{dM_{\Delta t}} } \right].
\end{equation}
Conditional Jensen's inequality and (\ref{mart_1}) imply that
\begin{equation}
\label{exp_bound_1}
\mathbb{E}_{M}\left[\frac{dP}{dM}\log{\frac{dP}{dM}}\middle| \mathcal{F}_{\Delta t} \right] \geq \zeta_{\Delta t}\log{\zeta_{\Delta t}}, \; M_{\Delta t}-\textrm{a.s.}\end{equation}
Taking expectations with respect to $M$ of both sides of (\ref{exp_bound_1}) we get that $\forall \Delta t >0$, 
\begin{equation*}
\begin{aligned}\mathbb{E}_{P}\left[\log{\frac{dP}{dM}}\right] & = \mathbb{E}_{M}\left[\frac{dP}{dM}\log{\frac{dP}{dM}}\right]\\ &= \mathbb{E}_{M}\left[ \mathbb{E}_{M}\left[\frac{dP}{dM}\log{\frac{dP}{dM}}\middle| \mathcal{F}_{\Delta t} \right]\right] \\ & \geq  \mathbb{E}_{M}\left[\frac{dP_{\Delta t}}{dM_{\Delta t}}\log{\frac{dP_{\Delta t}}{dM_{\Delta t}} } \right],  \end{aligned}
\end{equation*}
thus \begin{equation*}
\begin{aligned}\mathbb{E}_{P}\left[\log{\frac{dP}{dM}}\right] &\geq \limsup_{\Delta t \downarrow 0}\mathbb{E}_{M}\left[\frac{dP_{\Delta t}}{dM_{\Delta t}}\log{\frac{dP_{\Delta t}}{dM_{\Delta t}} } \right] \\ & =  \limsup_{\Delta t \downarrow 0} \mathbb{E}_{M_{\Delta t}}\left[\frac{dP_{\Delta t}}{dM_{\Delta t}}\log{\frac{dP_{\Delta t}}{dM_{\Delta t}} } \right] \\ 
& = \limsup_{\Delta t \downarrow 0} \mathbb{E}_{P_{\Delta t}}\left[\log{\frac{dP_{\Delta t}}{dM_{\Delta t}}}\right]\end{aligned}
\end{equation*}
The Radon-Nikodym Theorem guarantees that  $\frac{dP}{dM}$  is nonnegative and that $\frac{dP}{dM}\log{\frac{dP}{dM}}$ is $\mathcal{F}-$ measurable, thus 
\begin{equation}
    \begin{aligned}
     \liminf_{\Delta t \downarrow 0} \mathbb{E}_{P_{\Delta t}}\left[\log{\frac{dP_{\Delta t}}{dM_{\Delta t}}}\right] 
 &= \liminf_{\Delta t \downarrow 0}
 \mathbb{E}_{M}\left[\frac{dP_{\Delta t}}{dM_{\Delta t}}\log{\frac{dP_{\Delta t}}{dM_{\Delta t}} } \right]
\\& \geq \mathbb{E}_{M}\left[\frac{dP}{dM}\log{\frac{dP}{dM}}\right] \\&= \mathbb{E}_{P}\left[\log{\frac{dP}{dM}}\right]
    \end{aligned}
\end{equation}
as a consequence of the continuous time version of Fatou's Lemma and (\ref{thing_1}).
Now clearly $$\mathbb{E}_{P_{\Delta t}}\left[\log{\frac{dP_{\Delta t}}{dM_{\Delta t}}}\right] \rightarrow \mathbb{E}_{P}\left[\log{\frac{dP}{dM}}\right] \text{ as } \Delta t \downarrow 0.$$
\end{proof}
Let $\mathcal{F}_{X}^{[t_{0},T)}$ be the sub-$\sigma-$algebra of $\bigotimes_{t\in [t_{0},T)}\mathcal{X}$  defined by
\begin{equation}\label{imp_alg}
\mathcal{F}_{X}^{[t_{0},T)} = \bigcap_{\Delta t>0} \mathcal{F}_{\Delta t}^{[t_{0},T)}
\end{equation}
and observe that $\left(\mathcal{F}_{\Delta t}^{[t_{0},T)}\right)_{\Delta t>0}$ is a decreasing collection of $\sigma-$algebras due to (\ref{filtration}). Henceforth, when we write
$\mathbb{P}_{X\mid X, Y}^{(s,r)}[X_{t_0}^{T}\mid X_{t_{0} - s}^{t_{0}},  Y_{t_{0} - r}^{T}]\left( \cdot \right)$ or $\mathbb{P}_{X}^{(s)}[X_{t_0}^{T}\mid X_{t_{0} - s}^{t_{0}}\}]\left( \cdot \right)$, we are referring to the restriction of these measures to the $\sigma-$algebra $\mathcal{F}_{X}^{[t_{0},T)}$.  Furthermore, recall from (\ref{X_given_X}) and (\ref{X_given_XY}) that for all $A \in \mathcal{F}^{[t_{0},T)}_{{\Delta t}}$ and $ \omega \in \Omega$ we have that
 \begin{equation}\label{eq.9}
\begin{aligned}
 &\mathbb{P}_{X\mid X, Y}^{(s,r)}[X_{t_0}^{T}\mid X_{t_0 - s}^{t_{0}},  Y_{t_{0} - r}^{T}]\left( \omega \right)\Bigg |_{\mathcal{F}^{[t_{0},T)}_{{\Delta t}}}\!\!\left( A \right)=  \prod_{i = 0}^{\tau-1} \mathbb{P}_{X| \overleftarrow{X}, \overleftarrow{Y},i,\Delta t}^{(\omega),(k,l)}\left( \pi_{\left\langle T, i, \Delta t\right\rangle}(A)\right)
\end{aligned}
\end{equation}
and
\begin{equation}\label{eq.10}
\begin{aligned}
&\mathbb{P}_{X}^{(s)}[X_{t_{0}}^{T}\mid X_{t_{0} - s}^{t_{0}}]\left( \omega \right)\Bigg |_{\mathcal{F}^{[t_{0},T)}_{{\Delta t}}}\left( A \right)= \prod_{i = 0}^{\tau-1} \mathbb{P}_{X| \overleftarrow{X},i,\Delta t}^{(\omega),(k)}\big( \pi_{\left\langle T, i, \Delta t\right\rangle}(A)\big)
\end{aligned}
\end{equation}
where $k=\left\lfloor \frac{s}{\Delta t}\right\rfloor$ and $l=\left\lfloor \frac{r}{\Delta t}\right\rfloor$.  From now on, we will omit writing the projections in (\ref{eq.9}) and (\ref{eq.10}) to avoid cumbersome notation.
\begin{Notation}
For each $\omega \in \Omega, \Delta t>0$, we denote by $P_{\Delta t}^{(\omega)}$ and $M_{\Delta t}^{(\omega)}$ the measures 
$\mathbb{P}_{X\mid X, Y}^{(s,r)}[X_{t_0}^{T}\mid X_{t_0 - s}^{t_{0}},  Y_{t_{0} - r}^{T}]\left( \omega \right)\Bigg |_{\mathcal{F}^{[t_{0},T)}_{{\Delta t}}}$ and $\mathbb{P}_{X}^{(s)}[X_{t_0}^{T}\mid X_{t_0 - s}^{t_{0}}]\left( \omega \right)\Bigg |_{\mathcal{F}^{[t_{0},T)}_{{\Delta t}}},$ respectively.  It should be noted that these are measures on the measurable space $\left( \Sigma^{\tau}, \bigotimes^{\tau}\mathcal{X}\right).$
For $\Delta t >0$, let \begin{equation*}
\mathbb{T}_{Y \rightarrow X}^{(k,l),\Delta t}\left( \left\langle T,i,\Delta t \right\rangle\right)=\mathbb{E}_{\mathbb{P}}\left[ KL\left( {P}^{(k,l)}_{\Delta t}\Big| \Big|  
    {M}^{(k)}_{\Delta t}
     \right)\right]
\end{equation*}
for any $ i = 0,1,\ldots,\tau-1$
where $$P^{(k,l)}_{\Delta t} = \mathbb{P}^{(k,l)}_{\left\langle T,i,\Delta t \right\rangle}\left[X_{\left\langle T,i,\Delta t \right\rangle}\Big | \left(X_{\left\langle T,i+k+1,\Delta t \right\rangle}^{\left\langle T,i+1,\Delta t \right\rangle}\right), \left(Y_{\left\langle T,i+l+1,\Delta t \right\rangle}^{\left\langle T,i+1,\Delta t \right\rangle}\right)\right],$$
$${M}^{(k)}_{\Delta t} = \mathbb{P}^{(k)}_{\left\langle T,i,\Delta t \right\rangle}\left[X_{\left\langle T,i,\Delta t \right\rangle}\Big | \left(X_{\left\langle T,i+k+1,\Delta t \right\rangle}^{\left\langle T,i+1,\Delta t \right\rangle}\right) \right],$$
$$\left(X_{\left\langle T,i+k+1,\Delta t \right\rangle}^{\left\langle T, i+1, \Delta t\right\rangle}\right) = \left( 
X_{\left\langle T,i+k+1,\Delta t \right\rangle}, \dots, X_{\left\langle T,i+1,\Delta t \right\rangle}\right),$$ and $$\left(Y_{\left\langle T,i+l+1,\Delta t \right\rangle}^{\langle T, i+1, \Delta t\rangle}\right) = \left( 
Y_{\left\langle T,i+l+1,\Delta t \right\rangle}, \dots, Y_{\left\langle T,i+1,\Delta t \right\rangle}\right).$$
\end{Notation}

As a means of succinctly capturing all of the conditions which need hold to use Definitions \ref{disc_TE} and \ref{EPT_def}, we define a type of consistency between two processes dependent on the window lengths $r$ and $s$ and the interval $[t_{0},T)$.  This notion of consistency captures the conditions under which our main result, Theorem \ref{main}, is of utility.

\begin{definition}
Suppose $\mathbb{T} \subset  \mathbb{R}_{\geq 0}$ is a closed and bounded interval, $[t_{0}, T)\subset \mathbb{T}$, and $s,r >0$ are such that $\left(t_{0}-\max(s,r),T\right) \subset \mathbb{T}$. Suppose further that $X:=\{X_{t}\}_{t\in \mathbb{T}}$ and $Y:=\{Y_{t}\}_{t\in \mathbb{T}}$ are stochastic processes adapted to the filtered probability space $(\Omega, \mathcal{F}, \{\mathcal{F}_{t}\}_{t \in \mathbb{T}},\mathbb{P})$ such that for each $t \in \mathbb{T}, X_{t}$  and $Y_{t}$ are random variables taking values in the measurable space $(\Sigma, \mathcal{X})$, where $\Sigma$ is assumed to be a Polish space and $\mathcal{X}$ is a $\sigma-$algebra of subsets of $\Sigma$.
$Y$ is $(s,r)$-{\em consistent upon }$X$ {\em on} $[t_0,T)$ iff 
\begin{itemize}
\item[1.] $\forall \omega \in \Omega$ there exist measures $\mathbb{P}_{X}^{(s)}[X_{t_0}^{T}\mid X_{t_0 - s}^{t_{0}}\}]\left( \omega \right)$ and $\mathbb{P}_{X\mid X, Y}^{(s,r)}[X_{t_0}^{T}\mid X_{t_0 - s}^{t_{0}},  Y_{t_{0} - r}^{T}]\left( \omega \right)$ on the space $\left(\Omega_{X}^{[t_{0},T)},\mathcal{F}_{X}^{[t_{0},T)}\right)$ for which (\ref{X_given_X}) and (\ref{X_given_XY}) hold. 
\item[2.]$\exists \delta_{1}>0$ such that for all $ \Delta t\in(0,\delta_{1})$ and $ i = 0,1,\ldots,\tau-1$
\begin{itemize}
 \item[(a)] 
$\mathbb{P}_{X\mid \overleftarrow{X},\overleftarrow{Y},i,\Delta t}^{(\omega),(\lfloor \frac{s}{\Delta t}\rfloor,\lfloor \frac{r}{\Delta t}\rfloor)}\ll  \mathbb{P}_{X\mid \overleftarrow{X},i,\Delta t}^{(\omega),(\lfloor \frac{s}{\Delta t}\rfloor)}, \forall \omega\in\Omega.$

\item[(b)]
$\frac{d\mathbb{P}_{X| \overleftarrow{X},\overleftarrow{Y},i,\Delta t}^{(\omega),(k,l)}}{d\mathbb{P}_{X| \overleftarrow{X},i,\Delta t}^{(\omega),(k)}} \in L_{1}\left(\Sigma, \mathcal{X}, \mathbb{P}_{X| \overleftarrow{X},\overleftarrow{Y}, i,\Delta t}^{(\omega),(k,l)}\right), \forall \omega\in\Omega.$

\item[(c)]
$KL\left(P_{\Delta t}^{(\cdot)} \Bigg|\Bigg| M_{\Delta t}^{(\cdot)}\right)$ is $\mathbb{P}-$integrable.
\end{itemize}
\item[3.]  $ \mathbb{P}_{X\mid X, Y}^{(s,r)}[X_{t_0}^{T}\mid X_{t_o - s}^{t_{0}},  Y_{t_{0} - r}^{T}]\left( \omega \right)\ll \mathbb{P}_{X}^{(s)}[X_{t_0}^{T}\mid X_{t_o - s}^{t_{0}}]\left( \omega \right)$ for each $\omega \in \Omega$.
\end{itemize}
We call 1.- 3.~``consistency conditions''.
\end{definition}
We now present our main result.
\begin{theorem}\label{main}
Suppose $\mathbb{T} \subset  \mathbb{R}_{\geq 0}$ is a closed and bounded interval with $[t_{0},T)\subset \mathbb{T}$, $\Sigma$ is a Polish space and $s,r >0$ satisfy $\left(t_{0}-\max(s,r),T\right) \subset \mathbb{T}$. Suppose further that $X:=\{X_{t}\}_{t\in \mathbb{T}}$ and $Y:=\{Y_{t}\}_{t\in \mathbb{T}}$ are stochastic processes adapted to the filtered probability space $(\Omega, \mathcal{F}, \{\mathcal{F}_{t}\}_{t \in \mathbb{T}},\mathbb{P})$ such that for each $t \in \mathbb{T}, X_{t}$  and $Y_{t}$ are random variables taking values in the measurable state space $(\Sigma, \mathcal{X})$ and
that $Y$ is $(s,r)$-consistent upon $X$ on $[t_{0}, T)$.
If $\exists M, \delta_{2}>0\text{ such that }\forall \Delta t \in (0, \delta_{2}),$
\begin{equation}\label{bounder}
KL\left(P_{\Delta t}^{(\cdot)} \Bigg|\Bigg| M_{\Delta t}^{(\cdot)}\right)\leq M ,  \mathbb{P}-\textrm{a.s.},\end{equation} then 
 $$ {\mathcal{E}\mathcal{P}\mathcal{T}}_{Y \rightarrow X}^{(s,r)}\mid_{t_0}^{T} <\infty$$
iff 
\begin{equation}\label{Big_Fish} \lim_{\Delta t \downarrow 0}\left[ \sum_{i = 0}^{\tau-1}       \mathbb{T}_{Y \rightarrow X}^{(k,l),\Delta t}\left(  \left\langle T,i,\Delta t \right\rangle\right)\right] = {\mathcal{E}\mathcal{P}\mathcal{T}}_{Y \rightarrow X}^{(s,r)}\mid_{t_0}^{T}
\end{equation}
where $k=\lfloor \frac{s}{\Delta t}\rfloor$ and $l=\lfloor \frac{r}{\Delta t}\rfloor$.
\end{theorem}
\begin{proof}
$(\Rightarrow)$ Suppose ${\mathcal{E}\mathcal{P}\mathcal{T}}_{Y \rightarrow X}^{(s,r)}\mid_{t_0}^{T}<\infty$,
let $\delta = \min{\left\{\delta_{1}, \delta_{2}\right\}}$ and for each $\omega \in \Omega$ let
$$P^{(\omega)} = \mathbb{P}_{X\mid X, Y}^{(s,r)}[X_{t_0}^{T}\mid X_{t_{0} - s}^{t_{0}},  Y_{t_{0} - r}^{T}]\left( \omega \right)$$ and 
$$M^{(\omega)} = \mathbb{P}_{X}^{(s,r)}[X_{t_0}^{T}\mid X_{t_0 - s}^{t_{0}}]\left( \omega \right).$$  
If $\Delta t \in (0, \delta)$, then consistency condition 2(c) implies that $ KL\left( P_{\Delta t}^{(\omega)}|| M_{\Delta t}^{(\omega)} \right)$ is $\mathbb{P}$-integrable.  Since $\Sigma$ is $\sigma-$finite under both $\mathbb{P}_{X\mid \overleftarrow{X},\overleftarrow{Y},i,\Delta t}^{(\omega),(\left\lfloor \frac{s}{\Delta t}\right\rfloor,\left\lfloor \frac{r}{\Delta t}\right\rfloor)}$ and $\mathbb{P}_{X\mid \overleftarrow{X},i,\Delta t}^{(\omega),(\left\lfloor \frac{s}{\Delta t}\right\rfloor)}$ for each $ \omega\in\Omega$ and any $i =0,1,\dots, \tau-1$,  we have that the measurable space $\left( \Sigma^{\tau},  \bigotimes_{i=0}^{\tau-1}\mathcal{X}\right)$ is $\sigma-$finite under both $P_{\Delta t}^{(\omega)}$ and $M_{\Delta t}^{(\omega)}$ for each $\omega \in \Omega$, thus the RN-derivatives in (\ref{Big_Fish}) exist.  Furthermore, we get from Lemma \ref{fixedR-N} that
\begin{equation}\label{exp_is_sum_1}
\begin{aligned}
\mathbb{E}_{\mathbb{P}}\left[ KL\left( P_{\Delta t}^{(\cdot)}|| M_{\Delta t}^{(\cdot)} \right) \right]
& = \mathbb{E}_{\mathbb{P}}\left[ \mathbb{E}_{P_{\Delta t}^{(\omega)}}\left[\log{\frac{d\left( \prod_{i = 0}^{\tau-1} \mathbb{P}_{X| \overleftarrow{X}, \overleftarrow{Y},i,\Delta t}^{(\omega),(k,l)}\right)}{d \left( \prod_{i = 0}^{\tau-1} \mathbb{P}_{X| \overleftarrow{X},i,\Delta t}^{(\omega),(k)}\right)}}\right]\right]
\\ & = \mathbb{E}_{\mathbb{P}}\left[\mathbb{E}_{P_{\Delta t}^{(\omega)}}\left[ \log{\left(\prod_{i = 0}^{\tau-1}\left(\frac{d \mathbb{P}_{X| \overleftarrow{X}, \overleftarrow{Y},i,\Delta t}^{(\omega),(k,l)}}{d  \mathbb{P}_{X| \overleftarrow{X},i,\Delta t}^{(\omega),(k)}}\right)\right)}\right]\right]
\\& = \sum_{i = 0}^{\tau-1} \mathbb{E}_{\mathbb{P}}\left[ \mathbb{E}_{P_{\Delta t}^{(\omega)}} \left[ \log{\frac{d\mathbb{P}_{X|\overleftarrow{X},\overleftarrow{Y},i,\Delta t}^{(\omega),(k,l)}}{d\mathbb{P}_{X\mid \overleftarrow{X},i,\Delta t}^{(\omega),(k)}}}\right]\right]
\end{aligned}
\end{equation}
where $k = \left\lfloor \frac{s}{\Delta t}\right\rfloor$ and $l = \left\lfloor \frac{r}{\Delta t}\right\rfloor.$
Now for each $ \Delta t>0, i = 0,1,\cdots, \tau-1$ and $\omega \in \Omega$ let $$F_{i, \Delta t}^{\omega}\left( x_{0}, x_{1}, \cdots, x_{\tau-1}\right) = \log{\frac{d\mathbb{P}_{X|\overleftarrow{X},\overleftarrow{Y},i,\Delta t}^{(\omega),(k,l)}}{d\mathbb{P}_{X\mid \overleftarrow{X},i,\Delta t}^{(\omega),(k)}}}(x_{i})$$ for each $\tau$-tuple $\left( x_{0}, x_{1}, \cdots, x_{\tau-1}\right) \in \Sigma^{\tau}.$  Clearly, $F_{i, \Delta t}^{\omega}$ is $\Sigma^{\tau}-$measurable and furthermore $ P_{\Delta t}^{(\omega)}-$integrable due to Jensen's inequality since consistency condition 2(b) implies
 \begin{align*} \int_{\Sigma^{\tau}} \left[F_{i,\Delta t}^{\omega}\right]d P_{\Delta t}^{(\omega)} \leq \log\left( \int_{\Sigma^{\tau}} \left[ \frac{d\mathbb{P}_{X|\overleftarrow{X},\overleftarrow{Y},i,\Delta t}^{(\omega),(k,l)}}{d\mathbb{P}_{X\mid \overleftarrow{X},i,\Delta t}^{(\omega),(k)}}\right]d P_{\Delta t}^{(\omega)}\right)<\infty.
\end{align*}
Now we apply Fubini's Theorem and obtain 

\begin{align*}
&\sum_{i = 0}^{\tau-1} \mathbb{E}_{\mathbb{P}}\left[ \mathbb{E}_{P_{\Delta t}^{(\cdot)}} \left[ \log{\frac{d\mathbb{P}_{X|\overleftarrow{X},\overleftarrow{Y},i,\Delta t}^{(\cdot),(k,l)}}{d\mathbb{P}_{X\mid \overleftarrow{X},i,\Delta t}^{(\cdot),(k)}}}\right]\right] 
\\ & = \sum_{i = 0}^{\tau-1} \mathbb{E}_{\mathbb{P}}\left [ \int_{\Sigma^{\tau}} \left[ \log{\frac{d\mathbb{P}_{X|\overleftarrow{X},\overleftarrow{Y},i,\Delta t}^{(\cdot),(k,l)}}{d\mathbb{P}_{X\mid \overleftarrow{X},i,\Delta t}^{(\cdot),(k)}}}\right] \right . \left . d\left( \prod_{j = 0}^{\tau-1}\mathbb{P}_{X|\overleftarrow{X},\overleftarrow{Y},j,\Delta t}^{(\cdot),(k,l)}\right)\right ]
\\ 
& = \sum_{i = 0}^{\tau-1} \mathbb{E}_{\mathbb{P}}\left[ \int_{\Sigma^{\tau}} F_{i, \Delta t}^{(\cdot)}   d\left( \prod_{j = 0}^{\tau-1}\mathbb{P}_{X|\overleftarrow{X},\overleftarrow{Y},j,\Delta t}^{(\cdot),(k,l)}\right)\right]
\\& = \sum_{i = 0}^{\tau-1} \mathbb{E}_{\mathbb{P}}\left[ S_{i,\Delta t}\left( \int_{\Sigma} \log{\frac{d\mathbb{P}_{X|\overleftarrow{X},\overleftarrow{Y},i,\Delta t}^{(\cdot),(k,l)}}{d\mathbb{P}_{X\mid \overleftarrow{X},i,\Delta t}^{(\cdot),(k)}}} \hspace{1mm} d\mathbb{P}_{X|\overleftarrow{X},\overleftarrow{Y},i,\Delta t}^{(\cdot),(k,l)}\right)\right]
\\& =  \sum_{i = 0}^{\tau-1} \mathbb{T}_{Y \rightarrow X}^{(k,l), \Delta t}\left(  \left\langle T,i,\Delta t \right\rangle\right)
\end{align*}
where $S_{i,\Delta t} = \prod_{\substack{j = 0, j\neq i}}^{\tau -1} \int_{\Sigma}1 d \mathbb{P}_{X|\overleftarrow{X},\overleftarrow{Y},j,\Delta t}^{(\cdot),(k,l)}$ for $i = 0,1,\dots, \tau-1.$
Moreover, 
\begin{equation}\label{exp_is_sum}
\mathbb{E}_{\mathbb{P}}\left[ KL\left( P_{\Delta t}^{(\cdot)}|| M_{\Delta t}^{(\cdot)} \right) \right] =  \sum_{i = 0}^{\tau-1} \mathbb{T}_{Y \rightarrow X}^{(k,l), \Delta t}\left(  \left\langle T,i,\Delta t \right\rangle\right).
\end{equation}
Since ${\mathcal{E}\mathcal{P}\mathcal{T}}_{Y \rightarrow X}^{(s,r)}\mid_{t_0}^{T} <\infty$, we have 
$KL\left( P^{\left( \omega \right)} \big|\big| M^{\left( \omega \right)}\right) < \infty$ for all $\omega \in \Omega \backslash B$ for some $\mathbb{P}$-null set $B,$ which from Lemma \ref{KL_conv} implies that 
\begin{equation}\label{lemma_result}\mathbb{E}_{P_{\Delta t}^{(\cdot)}}\left[\log{\frac{dP_{\Delta t}^{(\cdot)}}{dM_{\Delta t}^{(\cdot)}}}\right] \rightarrow \mathbb{E}_{P^{(\cdot)}}\left[\log{\frac{dP^{(\cdot)}}{dM^{(\cdot)}}}\right]  \text{ as } \Delta t \downarrow 0, \mathbb{P}-\textrm{a.s.}\end{equation} 
Let
$$
g(\omega) = \left\{
\begin{array}{l l}
 \lim_{\Delta t \downarrow 0}\left( KL\left( P_{\Delta t}^{(\omega)}|| M_{\Delta t}^{(\omega)} \right)\right) & \quad  \omega \in \Omega \backslash B\\
0 & \quad \omega\in B\\ \end{array} \right. $$
and observe that $
g \in  L_{1}\left( \Omega, \mathcal{F}, \mathbb{P}\right)
$ and
\begin{equation}\label{L_1 thing}
\begin{aligned}\lim_{\Delta t \downarrow 0} KL\left( P_{\Delta t}^{(\cdot)}|| M_{\Delta t}^{(\cdot)} \right) = g , \mathbb{P} -\textrm{a.s.}
\end{aligned}
\end{equation}
Moreover, since $\mathbb{P}\left( \Omega\right) = 1$ we have
\begin{equation}\label{measure_conv}  KL\left( P_{\Delta t}^{(\cdot)}|| M_{\Delta t}^{(\cdot)} \right) \overset{\mathbb{P}}{\to} g \text{ as } \Delta t \downarrow 0.
\end{equation}
Now for each $ \epsilon, \Delta t>0$ and $ \omega \in \Omega$, define $h_{\Delta t}^{\epsilon}(\omega)$ by 
 $$ h_{\Delta t}^{\epsilon}(\omega)= \left\{
\begin{array}{l l}
KL\left( P_{\Delta t}^{(\omega)}|| M_{\Delta t}^{(\omega)} \right) & \quad  \Big|KL\left( P_{\Delta t}^{(\omega)}|| M_{\Delta t}^{(\omega)} \right) - g(\omega) \Big| < \epsilon\\
0 & \quad \mbox{otherwise}\\ \end{array} \right. $$
 and note that $h_{\Delta t}^{\epsilon}$ is nonnegative $\forall \epsilon, \Delta t>0$ due to Gibbs' inequality and converges in probability to $g$ since $\forall \eta >0$
 \begin{equation*}
\begin{aligned}
\mathbb{P}\left( \left\{ \Big| h_{\Delta t} - g \Big| \geq \eta \right\}\right)& \leq \mathbb{P}\left( \left\{\Big|KL\left( P_{\Delta t}^{(\cdot)}|| M_{\Delta t}^{(\cdot)} \right) - g \Big| \geq \eta \right\} \right)
\\& \qquad + \mathbb{P}\left( \left\{\Big|KL\left( P_{\Delta t}^{(\cdot)}|| M_{\Delta t}^{(\cdot)} \right) - g \Big| \geq \epsilon \right\} \right)  \rightarrow 0
\end{aligned}
\end{equation*}
 as $\Delta t \downarrow 0$. Let $\epsilon>0$ be arbitrary and observe that 
 \begin{equation}\label{norm_eq_spl}
 \begin{aligned}
 \lVert h_{\Delta t}^{\epsilon} - g \rVert_{L_{1}} &=
  \mathbb{E}_{\mathbb{P}}\left[ \Big|KL\left( P_{\Delta t}^{(\cdot)}|| M_{\Delta t}^{(\cdot)} \right) - g \Big| \chi_{\left\{\Big|KL\left( P_{\Delta t}^{(\cdot)}|| M_{\Delta t}^{(\cdot)} \right) - g \Big|<\epsilon\right\}}\right] \\&\qquad + \mathbb{E}_{\mathbb{P}}\left[  g \chi_{\left\{\Big|KL\left( P_{\Delta t}^{(\cdot)}|| M_{\Delta t}^{(\cdot)} \right) - g \Big|\geq \epsilon\right\}}\right]
 \\& < \epsilon  \mathbb{P}\left( \left\{\Big|KL\left( P_{\Delta t}^{(\cdot)}|| M_{\Delta t}^{(\cdot)} \right) - g \Big|<\epsilon\right\}\right)\\
 & \qquad + \mathbb{E}_{\mathbb{P}}\left[  g \chi_{\left\{\Big|KL\left( P_{\Delta t}^{(\cdot)}|| M_{\Delta t}^{(\cdot)} \right) - g \Big|\geq \epsilon\right\}}\right].
 \end{aligned}
 \end{equation}
 Since $g\in L_{1}\left( \Omega, \mathcal{F}, \mathbb{P}\right)$ and $KL\left( P_{\Delta t}^{(\cdot)}|| M_{\Delta t}^{(\cdot)} \right) \overset{\mathbb{P}}{\to} g\hspace{.1in} \text{as } \Delta t \downarrow 0,$ we have
 $$
 \lim_{\Delta t \downarrow 0} \mathbb{E}_{\mathbb{P}}\left[  g \chi_{\left\{\Big|KL\left( P_{\Delta t}^{(\cdot)}|| M_{\Delta t}^{(\cdot)} \right) - g \Big|\geq \epsilon\right\}}\right] = 0.
 $$
Now since $\mathbb{P}\left( \left\{\Big|KL\left( P_{\Delta t}^{(\cdot)}|| M_{\Delta t}^{(\cdot)} \right) - g \Big|<\epsilon\right\}\right) \rightarrow 1$ as $\Delta t \downarrow 0,$ we obtain  
\begin{equation*}\label{L_1_bound}
 \lim_{\Delta t \downarrow 0}\lVert h_{\Delta t}^{\epsilon} - g \rVert_{L_{1}} \leq \epsilon
\end{equation*}
from (\ref{norm_eq_spl}) and thus $ \lim_{\epsilon \downarrow 0}\lim_{\Delta t \downarrow 0}\lVert h_{\Delta t}^{\epsilon} - g \rVert_{L_{1}} = 0$
since $\epsilon > 0$ was arbitrary.
In particular, 
\begin{equation}\label{lim_h}
 \lim_{\epsilon \downarrow 0}\lim_{\Delta t \downarrow 0}\mathbb{E}_{\mathbb{P}}\left[ h_{\Delta t}^{\epsilon} \right]= \mathbb{E}_{\mathbb{P}}\left[g\right] = \mathbb{E}_{\mathbb{P}}\left[  \lim_{\Delta t \downarrow 0} KL\left( P_{\Delta t}^{(\cdot)}|| M_{\Delta t}^{(\cdot)} \right) \right].
\end{equation}
We now show that 
\begin{equation}\label{punch_line}
 \lim_{\epsilon \downarrow 0}\lim_{\Delta t \downarrow 0}\mathbb{E}_{\mathbb{P}}\left[ h_{\Delta t}^{\epsilon} \right] = \lim_{\Delta t \downarrow 0 }\mathbb{E}_{\mathbb{P}}\left[ KL\left( P_{\Delta t}^{(\cdot)}|| M_{\Delta t}^{(\cdot)} \right) \right].
\end{equation}
Note that 
\begin{equation*}
 \lim_{\epsilon \downarrow 0}\lim_{\Delta t \downarrow 0}\mathbb{E}_{\mathbb{P}}\left[ \alpha _{\Delta t}^{\epsilon}\right] = 0 \implies
 \lim_{\epsilon \downarrow 0}\lim_{\Delta t \downarrow 0}\mathbb{E}_{\mathbb{P}}\left[ h_{\Delta t}^{\epsilon} -  KL\left( P_{\Delta t}^{(\cdot)}|| M_{\Delta t}^{(\cdot)} \right)\right] = 0
\end{equation*}
where $$\alpha _{\Delta t}^{\epsilon}(\omega) = KL\left( P_{\Delta t}^{(\omega)}|| M_{\Delta t}^{(\omega)} \right) \chi_{ \left\{\Big|KL\left( P_{\Delta t}^{(\cdot)}|| M_{\Delta t}^{(\cdot)} \right) - g \Big| \geq \epsilon\right\} }(\omega)$$
for $\epsilon, \Delta t>0,$ and $\omega\in\Omega.$  Fix $\epsilon>0$ and note that (\ref{bounder}) implies \begin{equation}\label{bound}
0 \leq \mathbb{E}_{\mathbb{P}}\left[ \alpha _{\Delta t}^{\epsilon}\right] \leq M \mathbb{P}\left( \left\{\Big|KL\left( P_{\Delta t}^{(\cdot)}|| M_{\Delta t}^{(\cdot)} \right) - g \Big|\geq \epsilon\right\}\right]
\end{equation}
$\forall \Delta t \in(0,\delta)$.  Due to (\ref{measure_conv}), the RHS of 
(\ref{bound}) converges to $0$ as $\Delta t\downarrow 0$, thus 
\begin{equation}\label{point}
 \lim_{\epsilon \downarrow 0}\lim_{\Delta t \downarrow 0}\mathbb{E}_{\mathbb{P}}\left[ \alpha _{\Delta t}^{\epsilon}\right] = 0
\end{equation}
so 
\begin{equation}\label{lim_diff}
 \lim_{\epsilon \downarrow 0}\lim_{\Delta t \downarrow 0}\mathbb{E}_{\mathbb{P}}\left[ h_{\Delta t}^{\epsilon} -  KL\left( P_{\Delta t}^{(\cdot)}|| M_{\Delta t}^{(\cdot)} \right)\right] = 0.
\end{equation}
Now from (\ref{lim_h}) and (\ref{lim_diff}) we have that $ \lim_{\Delta t \downarrow 0 }\mathbb{E}_{\mathbb{P}}\left[ KL\left( P_{\Delta t}^{(\cdot)}|| M_{\Delta t}^{(\cdot)} \right) \right]\text{ exists}$ since $$\mathbb{E}_{\mathbb{P}}\left[ KL\left( P_{\Delta t}^{(\cdot)}|| M_{\Delta t}^{(\cdot)} \right) \right] = \mathbb{E}_{\mathbb{P}}\left[ h_{\Delta t}^{\epsilon} \right] -\left[ \mathbb{E}_{\mathbb{P}}\left[ h_{\Delta t}^{\epsilon} -  KL\left( P_{\Delta t}^{(\cdot)}|| M_{\Delta t}^{(\cdot)} \right)\right]\right].$$
Hence \begin{equation*}
\begin{aligned}
0 &= \lim_{\epsilon \downarrow 0}\lim_{\Delta t \downarrow 0}\mathbb{E}_{\mathbb{P}}\left[ h_{\Delta t}^{\epsilon} -  KL\left( P_{\Delta t}^{(\cdot)}|| M_{\Delta t}^{(\cdot)} \right)\right]
\\& = \lim_{\epsilon \downarrow 0}\lim_{\Delta t \downarrow 0}\left(\mathbb{E}_{\mathbb{P}}\left[ h_{\Delta t}^{\epsilon} \right]-  \mathbb{E}_{\mathbb{P}}\left[KL\left( P_{\Delta t}^{(\cdot)}|| M_{\Delta t}^{(\cdot)} \right)\right]\right)
\\& = \lim_{\epsilon \downarrow 0}\lim_{\Delta t \downarrow 0}\left(\mathbb{E}_{\mathbb{P}}\left[ h_{\Delta t}^{\epsilon} \right]\right) -  \lim_{\epsilon \downarrow 0}\lim_{\Delta t \downarrow 0}\left(\mathbb{E}_{\mathbb{P}}\left[KL\left( P_{\Delta t}^{(\cdot)}|| M_{\Delta t}^{(\cdot)} \right)\right]\right)
\\ & = \lim_{\epsilon \downarrow 0}\lim_{\Delta t \downarrow 0}\left(\mathbb{E}_{\mathbb{P}}\left[ h_{\Delta t}^{\epsilon} \right]\right) - \lim_{\Delta t \downarrow 0}\left(\mathbb{E}_{\mathbb{P}}\left[KL\left( P_{\Delta t}^{(\cdot)}|| M_{\Delta t}^{(\cdot)} \right)\right]\right)
\end{aligned}
\end{equation*}
proving (\ref{punch_line}).  Now we have \begin{equation}\label{end}
\lim_{\Delta t \downarrow 0}\mathbb{E}_{\mathbb{P}}\left[ KL\left( P_{\Delta t}^{(\cdot)}|| M_{\Delta t}^{(\cdot)} \right) \right] = \mathbb{E}_{\mathbb{P}}\left[
 \lim_{\Delta t \downarrow 0} KL\left( P_{\Delta t}^{(\cdot)}|| M_{\Delta t}^{(\cdot)} \right)
\right]
\end{equation}
from which the result follows, since
\begin{align*}
{\mathcal{E}\mathcal{P}\mathcal{T}}_{Y \rightarrow X}^{(s,r)}\mid_{t_0}^{T}& = \mathbb{E}_{\mathbb{P}}\left[KL\left( P^{(\cdot)}|| M^{(\cdot)} \right)\right]
\\& =  \mathbb{E}_{\mathbb{P}}\left[ \lim_{\Delta t \downarrow 0} KL\left( P_{\Delta t}^{(\cdot)}|| M_{\Delta t}^{(\cdot)} \right)\right]
\\ & = \lim_{\Delta t \downarrow 0}\mathbb{E}_{\mathbb{P}}\left[ KL\left( P_{\Delta t}^{(\cdot)}|| M_{\Delta t}^{(\cdot)} \right) \right] 
\\& =  \lim_{\Delta t \downarrow 0}\left[ \sum_{i = 0}^{\tau-1}       \mathbb{T}_{Y \rightarrow X}^{(k,l),\Delta t}\left(  \left\langle T,i,\Delta t \right\rangle\right)\right].
\end{align*}
$(\Leftarrow)$ Suppose towards a contradiction 
$$
\lim_{\Delta t \downarrow 0}\left[ \sum_{i = 0}^{\tau-1}       \mathbb{T}_{Y \rightarrow X}^{(k,l),\Delta t}\left(  \left\langle T,i,\Delta t \right\rangle\right)\right] = {\mathcal{E}\mathcal{P}\mathcal{T}}_{Y \rightarrow X}^{(s,r)}\mid_{t_0}^{T} = \infty.
$$
Then 
\begin{equation*}\label{inf_lim}
 \lim_{\Delta t \downarrow 0}\mathbb{E}_{\mathbb{P}}\left[ KL\left( P_{\Delta t}^{(\cdot)}|| M_{\Delta t}^{(\cdot)} \right) \right] = \infty,
\end{equation*}
thus $\exists \delta_{3}>0$ such that $\Delta t \in (0,\delta_{3}) \implies \mathbb{E}_{\mathbb{P}}\left[ KL\left( P_{\Delta t}^{(\cdot)}|| M_{\Delta t}^{(\cdot)} \right) \right] > M.$
From (\ref{bounder}),
\begin{equation*}\label{bound_3}KL\left( P_{\Delta t}^{(\cdot)}|| M_{\Delta t}^{(\cdot)} \right)\leq M  , \mathbb{P}-\textrm{a.s.}
\end{equation*}
$\forall \Delta t \in (0, \delta_{2})$, hence $$M < \mathbb{E}_{\mathbb{P}}\left[ KL\left( P_{\Delta t}^{(\cdot)}|| M_{\Delta t}^{(\cdot)} \right) \right]\leq \mathbb{E}_{\mathbb{P}}[M] = M,$$
$\forall \Delta t \in (0, \min{\{\delta_{3}, \delta_{2}\}})$.  This is a contradiction and the proof is complete.
\end{proof}
Due to the following corollary, one can conclude the ``only if'' part of Theorem \ref{main} under a weaker version of (\ref{bounder}).
\begin{corollary}\label{l1_bound}
Let $\mathbb{T} \subset  \mathbb{R}_{\geq 0}$ be an interval and $[t_{0},T)\subset \mathbb{T}$ and $s,r >0$ be such that $\left(t_{0}-\max(s,r),T\right) \subset \mathbb{T}$. Suppose $X:=\{X_{t}\}_{t\in \mathbb{T}}$ and $Y:=\{Y_{t}\}_{t\in \mathbb{T}}$ are stochastic processes adapted to the filtered probability space $(\Omega, \mathcal{F}, \{\mathcal{F}_{t}\}_{t \in \mathbb{T}},\mathbb{P})$ such that for each $t \in \mathbb{T}, X_{t}$  and $Y_{t}$ are random variables taking values in the measurable state space $(\Sigma, \mathcal{X})$ and
Y is $(s,r)-$SPL consistent upon $X$ on $[t_{0}, T)$.
If there exist $\eta\in L_{1}\left( \Omega, \mathcal{F},\mathbb{P}\right)$ and $\delta_{2}>0$ such that $KL\left(P_{\Delta t}^{(\cdot)} || M_{\Delta t}^{(\cdot)}\right)\leq \eta(\cdot) , \mathbb{P}-\textrm{a.s.}$ $\forall\Delta t \in (0, \delta_{2}) $ and
 ${\mathcal{E}\mathcal{P}\mathcal{T}}_{Y \rightarrow X}^{(s,r)}\mid_{t_0}^{T}<\infty,$
then
$$
\lim_{\Delta t \downarrow 0}\left[ \sum_{i = 0}^{\tau-1}       \mathbb{T}_{Y \rightarrow X}^{(k,l),\Delta t}\left(  \left\langle T,i,\Delta t \right\rangle\right)\right] = {\mathcal{E}\mathcal{P}\mathcal{T}}_{Y \rightarrow X}^{(s,r)}\mid_{t_0}^{T}
$$
where $k=\lfloor \frac{s}{\Delta t}\rfloor$ and $l=\lfloor \frac{r}{\Delta t}\rfloor$.
\end{corollary}
\begin{proof}
We need only show that (\ref{point}) in the proof of the forward direction of Theorem \ref{main} is still true.  Since $\eta\in L_{1}\left( \Omega, \mathcal{F},\mathbb{P}\right)$, for $\epsilon >0$ we have that  \begin{equation*}
    \begin{aligned}
    \mathbb{E}_{\mathbb{P}}\left[\alpha _{\Delta t}^{\epsilon}\right] &= \mathbb{E}_{\mathbb{P}}\left[KL\left( P_{\Delta t}^{(\cdot)}|| M_{\Delta t}^{(\cdot)} \right) \chi_{ \left\{\Big|KL\left( P_{\Delta t}^{(\cdot)}|| M_{\Delta t}^{(\cdot)} \right) - g \Big| \geq \epsilon\right\} }\right] \\& \leq \mathbb{E}_{\mathbb{P}}\left[ \eta  \chi_{ \left\{\Big|KL\left( P_{\Delta t}^{(\cdot)}|| M_{\Delta t}^{(\cdot)} \right) - g \Big| \geq \epsilon\right\} }\right] \rightarrow 0
    \end{aligned}
\end{equation*}
as $\Delta t \downarrow 0$ due to (\ref{measure_conv}).
\end{proof}
The following corollary of Theorem \ref{main} is a key result because it will be used in an application to be explored later in Section \ref{Application: Lagged Poisson point process}.  The conditions in Theorem \ref{main} may be too strong to apply to some common situations.  The following weakens these conditions at the cost of the equivalence between the hypotheses and conclusion in Theorem \ref{main}.
\begin{corollary}\label{PPPusecase}
Let $\mathbb{T} \subset  \mathbb{R}_{\geq 0}$ be a closed and bounded interval, $[t_{0},T)\subset \mathbb{T}$, and $s,r >0$ be such that $\left(t_{0}-\max(s,r),T\right) \subset \mathbb{T}$. Suppose $X:=\{X_{t}\}_{t\in \mathbb{T}}$ and $Y:=\{Y_{t}\}_{t\in \mathbb{T}}$ are stochastic processes adapted to the filtered probability space $(\Omega, \mathcal{F}, \{\mathcal{F}_{t}\}_{t \in \mathbb{T}},\mathbb{P})$ such that for each $t \in \mathbb{T}, X_{t}$  and $Y_{t}$ are random variables taking values in the measurable state space $(\Sigma, \mathcal{X})$ and
Y is $(s,r)-$SPL consistent upon $X$ on $[t_{0}, T)$.
If there exists $\gamma>0$ such that
\begin{equation}\label{PPP_help}
\lim_{\Delta t\downarrow 0 }\mathbb{P}\left(B_{\Delta t,\gamma}\right) = 1
\end{equation}
where
\begin{equation}\label{bounder_2}
B_{\Delta t,\gamma} = \left\{\omega\in \Omega : \Delta t' \in (0, \Delta t) \implies  
KL\left(P_{\Delta t'}^{(\omega)} \Bigg|\Bigg| M_{\Delta t'}^{(\omega)}\right) \leq \gamma\right\}
\end{equation}
for $\Delta t,\lambda>0$
and
 $ {\mathcal{E}\mathcal{P}\mathcal{T}}_{Y \rightarrow X}^{(s,r)}\mid_{t_0}^{T}<\infty$,
then
$$
\lim_{\Delta t \downarrow 0}\left[ \sum_{i = 0}^{\tau-1}       \mathbb{T}_{Y \rightarrow X}^{(k,l),\Delta t}\left(  \left\langle T,i,\Delta t \right\rangle\right)\right] = {\mathcal{E}\mathcal{P}\mathcal{T}}_{Y \rightarrow X}^{(s,r)}\mid_{t_0}^{T}
$$
where $k=\left\lfloor \frac{s}{\Delta t}\right\rfloor$ and $l=\left\lfloor \frac{r}{\Delta t}\right\rfloor$.
\end{corollary}
\begin{proof}
As in Corollary \ref{l1_bound}, it suffices to show that (\ref{point}) holds whenever both (\ref{PPP_help}) and ${\mathcal{E}\mathcal{P}\mathcal{T}}_{Y \rightarrow X}^{(s,r)}\mid_{t_0}^{T}<\infty$ hold.  Observe that 
\begin{equation*}
\begin{aligned}
 &\mathbb{E}_{\mathbb{P}}\left[KL\left( P_{\Delta t}^{(\cdot)}|| M_{\Delta t}^{(\cdot)} \right) \chi_{ \left\{\Big|KL\left( P_{\Delta t}^{(\cdot)}|| M_{\Delta t}^{(\cdot)} \right) - g \Big| \geq \epsilon\right\} \bigcap B_{\Delta t}}\right]\\
& \leq \mathbb{E}_{\mathbb{P}}\left[  \gamma \chi_{\left\{\Big|KL\left( P_{\Delta t}^{(\cdot)}|| M_{\Delta t}^{(\cdot)} \right) - g \Big| \geq \epsilon\right\}}\right] \rightarrow 0 
\end{aligned}
\end{equation*} 
as $\Delta t \downarrow 0$, since clearly $\gamma\in L_{1}\left( \Omega, \mathcal{F}, \mathbb{P}\right)$.  Let $\tau' = \left\lfloor\frac{T}{\Delta t'}\right\rfloor- \left\lfloor\frac{t_{0}}{\Delta t'}\right\rfloor$ for $\Delta t'>0$ and observe that since $ {\mathcal{E}\mathcal{P}\mathcal{T}}_{Y \rightarrow X}^{(s,r)}\mid_{t_0}^{T}<\infty,$ Lemma \ref{KL_conv} implies that 
$$
KL\left( \prod_{i = 0}^{\tau'-1} \mathbb{P}_{X| \overleftarrow{X}, \overleftarrow{Y},i,\Delta t'}^{(\cdot),(k,l)} \Bigg|\Bigg| \prod_{i = 0}^{\tau'-1} \mathbb{P}_{X| \overleftarrow{X},i,\Delta  t'}^{(\cdot),(k)}\right)\in L_{1}\left(\Omega, \mathcal{F}, \mathbb{P} \right)
$$ for all $\Delta t'>0$ in a small enough neighborhood of 0;  moreover,
\begin{equation*}
\begin{aligned}
&\mathbb{E}_{\mathbb{P}}\left[KL\left( P_{\Delta t}^{(\cdot)}|| M_{\Delta t}^{(\cdot)} \right) \chi_{ \left\{\Big|KL\left( P_{\Delta t}^{(\cdot)}|| M_{\Delta t}^{(\cdot)} \right) - g \Big| \geq \epsilon\right\} \bigcap \overline{B_{\Delta t}}}\right]\rightarrow 0
\end{aligned}
\end{equation*}
as $\Delta t \downarrow 0$ since $\mathbb{P}\left(\overline{B_{\Delta t}} \right)\rightarrow 0$.  Now for any $\epsilon>0$,
\begin{align*}
    \mathbb{E}_{\mathbb{P}}\left[\alpha _{\Delta t}^{\epsilon}\right] &= \mathbb{E}_{\mathbb{P}}\left[KL\left( P_{\Delta t}^{(\cdot)}|| M_{\Delta t}^{(\cdot)} \right) \chi_{ \left\{\Big|KL\left( P_{\Delta t}^{(\cdot)}|| M_{\Delta t}^{(\cdot)} \right) - g \Big| \geq \epsilon\right\} }\right]
    \\& = \mathbb{E}_{\mathbb{P}}\left[KL\left( P_{\Delta t}^{(\cdot)}|| M_{\Delta t}^{(\cdot)} \right) \chi_{ \left\{\Big|KL\left( P_{\Delta t}^{(\cdot)}|| M_{\Delta t}^{(\cdot)} \right) - g \Big| \geq \epsilon\right\} \bigcap B_{\Delta t}}\right]  \\ &\qquad  + \mathbb{E}_{\mathbb{P}}\left[KL\left( P_{\Delta t}^{(\cdot)}|| M_{\Delta t}^{(\cdot)} \right) \chi_{ \left\{\Big|KL\left( P_{\Delta t}^{(\cdot)}|| M_{\Delta t}^{(\cdot)} \right) - g \Big| \geq \epsilon\right\} \bigcap \overline{B_{\Delta t}}}\right] \rightarrow 0
\end{align*}
as $ \Delta t\downarrow 0$, proving the corollary.
\end{proof}

\section{Application: lagged Poisson point process}
\label{Application: Lagged Poisson point process}

Below, we provide an example of two processes which satisfy (\ref{bounder}) of Theorem \ref{main} under a certain assumption on $r$. 
In the following example, we consider TE from a time-lagged version of the counting process of a Time-Homogeneous Poisson Point Process (THPPP) to itself, a case through which we demonstrate the applicability of our results.
	
Suppose $[t_{0}, T)\subset \mathbb{T} \subset\mathbb{R},X = \left(X_{t} \right)_{t\in\mathbb{T}}$ is the counting process of a THPPP with intensity $\lambda$.  Suppose further that $\epsilon >0$ and $Y=\left( Y_{t}\right)_{t\in\mathbb{T}}$ $Y_{t} = X_{t+\epsilon}, \forall t \geq -\epsilon.$ If $X$ is the counting process with intensity $\lambda >0$ of a THPPP $\psi:=\left( T_{n}\right)_{n\geq 1}$, then $Y$ is also a counting process of a THPPP with intensity $\lambda>0$, specifically that of the point process $\psi':=\left( T_{n}-\epsilon \right)_{n\geq 1}$.  Note that the state space of $X_{t}$ is the natural numbers for any $t\in[t_{0},T)$; a Polish space with discrete metric.  For any $\omega \in \Omega$, $\Delta t>0$ and $i = 0,1,\ldots,\tau-1$ we have
	\begin{align*}
	&\mathbb{P}_{\left\langle T, i, \Delta t\right\rangle}\left( X_{\left\langle T, i, \Delta t\right\rangle}   \middle|  \left(X_{\left\lfloor\frac{T}{\Delta t}\right\rfloor \Delta t-(i+k+1)\Delta t}^{\left\langle T,i+1,\Delta t \right\rangle}\right) \right)(\omega)\left( b_{\left\langle T, i, \Delta t\right\rangle}\right)\\
	& = \mathbb{P}_{\left\langle T, i, \Delta t\right\rangle}\left( X_{\left\langle T, i, \Delta t\right\rangle} \middle|  X_{\left\langle T,i+1,\Delta t \right\rangle} \right)(\omega)\left( b_{\left\langle T, i, \Delta t\right\rangle}\right)\\
	& = \mathbb{P}\left(X_{\left\langle T, i, \Delta t\right\rangle} - X_{\left\langle T,i+1,\Delta t \right\rangle } = b_{\left\langle T,i,\Delta t \right\rangle}  - X_{\left\langle T,i+1,\Delta t \right\rangle }(\omega)\right)\\
	& = e^{-\lambda \Delta t }\left(\frac{\left(\lambda \Delta t\right) ^{b_{\left\langle T,i,\Delta t \right\rangle}  - X_{\left\langle T,i+1,\Delta t \right\rangle }(\omega)}}{\left(b_{\left\langle T,i,\Delta t \right\rangle}  - X_{\left\langle T,i+1,\Delta t \right\rangle }(\omega)\right)!}\right)\\
	& = \pois\left( \lambda \Delta t; b_{\left\langle T,i,\Delta t \right\rangle}  - X_{\left\langle T,i+1,\Delta t \right\rangle }(\omega)\right)
	\end{align*}
	where $\pois\left(x,n \right) = \frac{e^{-x}x^n}{n!}$ for $x>0$ and integers $n\geq 0.$
	
	Suppose that $[t_{0}-\max{(\epsilon,s)},T)\subset \mathbb{T}$ and  $0<r<\epsilon$.  Then $\exists \Delta t^{\star}>0$ such that $0<j\Delta t^{\star}<\epsilon$, $\forall j = 1,2,\cdots, \left\lfloor\frac{r}{\Delta t^{\star}} \right\rfloor$. Letting $L = \left\lfloor\frac{r}{\Delta t^{\star}} \right\rfloor$ we get that 
	\begin{equation}
	\begin{aligned}
	&\mathbb{P}_{\left\langle T,i,\Delta t^{\star} \right\rangle}\left( X_{\left\langle T,i,\Delta t^{\star} \right\rangle}\middle| \left(X_{\left\langle T,i+k+1,\Delta t^{\star} \right\rangle}^{\left\langle T,i+1,\Delta t^{\star} \right\rangle}\right),  \left(Y_{               \left\langle T,i+L+1,\Delta t^{\star} \right\rangle }^{\left\langle T,i+1,\Delta t^{\star} \right\rangle} \right)\right)(\omega)
	\\
	& = \mathbb{P}_{\left\langle T,i,\Delta t^{\star} \right\rangle}\left( X_{\left\langle T,i,\Delta t^{\star} \right\rangle}\middle| X_{\left\langle T,i+1,\Delta t^{\star} \right\rangle},  X_{               \left\langle T,i+L,\Delta t^{\star} \right\rangle +\epsilon }\right)(\omega)(\cdot)\\
	& = \frac{\pois\left( \lambda(\epsilon - L\Delta t^{\star});X_{\left\langle T,i+L,\Delta t^{\star} \right\rangle +\epsilon}(\omega) - b_{\left\langle T,i,\Delta t^{\star} \right\rangle} \right) \cdot p_{\Delta t^{{\star}},i,\omega}}{\pois\left( \lambda \left((1-L)\Delta t^{\star} +\epsilon \right); X_{\left\langle T,i+L,\Delta t^{\star} \right\rangle +\epsilon}(\omega) -  X_{\left\langle T,i+1,\Delta t^{\star} \right\rangle }(\omega)\right)}
	\\& =: f_{\epsilon, \lambda, \omega}(\Delta t^{\star}, i, b_{\left\langle T,i,\Delta t^{\star} \right\rangle} )
	\end{aligned}
	\end{equation}
	where we define $p_{\Delta t^{\star},i,\omega} = \pois\left( \lambda\Delta t^{\star}; b_{\left\langle T,i,\Delta t^{\star} \right\rangle} -  X_{\left\langle T,i+1,\Delta t^{\star} \right\rangle}(\omega)\right)$.  Let $a_{\omega,i} = X_{\left\langle T,i+1,\Delta t^{\star} \right\rangle}(\omega)$ and $c_{\omega,i} = X_{\left\langle T,i+L,\Delta t^{\star} \right\rangle +\epsilon}(\omega)$ and observe that for any $i = 0,1,\ldots,\left\lfloor\frac{T}{\Delta t^{\star}}\right\rfloor - \left\lfloor\frac{t_{0}}{\Delta t^{\star}}\right\rfloor-1$ we have that
	\begin{small}
	\begin{align*}
	&KL\left(\mathbb{P}_{X| \overleftarrow{X}, \overleftarrow{Y},i,\Delta t^{\star}}^{(\omega),(k,l)} \middle| \middle| \mathbb{P}_{X| \overleftarrow{X},i,\Delta t^{\star}}^{(\omega),(k)} \right)\\
	& = \sum_{b \in Range\left( X_{T-i\Delta t^{\star}}\right)}
	 f_{\epsilon, \lambda, \omega}(\Delta t^{\star}, i, b )\log\frac{f_{\epsilon, \lambda, \omega}(\Delta t^{\star}, i, b )}{\pois\left( \lambda \Delta t^{\star}; b  - X_{\left\langle T,i+1,\Delta t^{\star} \right\rangle }(\omega)\right)} \\
	& = \sum_{a_{\omega}\leq b \leq c_{\omega} }
	 f_{\epsilon, \lambda, \omega}(\Delta t^{\star}, i, b )\log\frac{f_{\epsilon, \lambda, \omega}(\Delta t^{\star}, i, b )}{\pois\left( \lambda \Delta t^{\star}; b  - X_{\left\langle T,i+1,\Delta t^{\star} \right\rangle }(\omega)\right)} \\
	& = \sum_{0 \leq b \leq c_{\omega} - a_{\omega} }
	 f_{\epsilon, \lambda, \omega}(\Delta t^{\star}, i, a_{\omega}+b )\log\frac{\pois\left( \lambda(\epsilon - L\Delta t^{\star});c_{\omega} -a_{\omega}- b \right)}{  \pois\left( \lambda \left((1-L)\Delta t^{\star} +\epsilon \right);c_{\omega} - a_{\omega}\right)  }\\
	& = \sum_{0 \leq b \leq c_{\omega} - a_{\omega} }\binom{c_{\omega} - a_{\omega}}{b}\left( \frac{\Delta t^{\star}}{\epsilon - L\Delta t^{\star}}\right)^{b}\left( \frac{\epsilon -L\Delta t^{\star}}{\epsilon +(1-L)\Delta t^{\star}}\right)^{c_{\omega} - a_{\omega}}\times \\ &\qquad \left[ \lambda \Delta t^{\star} +\log\left(\frac{(c_{\omega}-a_{\omega})^{\underline{b}}}{\left(\lambda(\epsilon - L \Delta t^{\star}) \right)^{b}} \right) - (c_{\omega} - a_{\omega})\log\left(1+ \frac{\Delta t^{\star}}{\epsilon - L\Delta t^{\star}}\right)\right] \\
	& = \left[\eta\left( \left( \frac{\epsilon -L\Delta t^{\star}}{\epsilon +(1-L)\Delta t^{\star}}\right)^{c_{\omega} - a_{\omega}}\right) + \lambda \Delta t^{\star}\left( \frac{\epsilon -L\Delta t^{\star}}{\epsilon +(1-L)\Delta t^{\star}}\right)^{c_{\omega} - a_{\omega}} \right]\sum_{b=0}^{ c_{\omega} - a_{\omega} }\zeta_{\Delta t^{\star}}(b) \\ & \qquad+ \left( \left( \frac{\epsilon -L\Delta t^{\star}}{\epsilon +(1-L)\Delta t^{\star}}\right)^{c_{\omega} - a_{\omega}}\right) \sum_{b=0}^{ c_{\omega} - a_{\omega} }\zeta_{\Delta t^{\star}}(b)\log\left( \frac{\left(c_{\omega} - a_{\omega} \right)^{\underline{b}}}{\lambda^{b}(\epsilon - \lambda \Delta t^{\star})^{b}}\right)
	\end{align*}
	\end{small}
where $\zeta_{\Delta t^{\star}}(b) = \binom{c_{\omega} - a_{\omega}}{b}\left( \frac{\Delta t^{\star}}{\epsilon - L\Delta t^{\star}}\right)^{b}$ for $0\leq b \leq c_{\omega} - a_{\omega}$, $\eta(x) = x\log(x)$ for $x>0$ and  $x^{\underline{b}} := b! \binom{x}{b}$ denotes the $b${\em -th falling factorial of} $x$.
We suppose now that $\forall \omega \in \Omega, \exists \Delta t_{\omega}>0$ such that $X_{t+{\Delta t_{\omega}}}(\omega) - X_{t}(\omega) \leq 1$ for all $ t\in [t_{0},T)$; that is, there is no more than one event in any interval of length $\Delta t_{\omega}$. 
Under this assumption, if $\omega \in \Omega$ and $0 <\Delta t <\min\left\{\Delta t_{\omega}, \Delta t^{\star}\right\}$, then
\begin{small}
\begin{align*}
&KL\left(\mathbb{P}_{X| \overleftarrow{X}, \overleftarrow{Y},i,\Delta t}^{(\omega),(k,l)} \middle| \middle| \mathbb{P}_{X| \overleftarrow{X},i,\Delta t}^{(\omega),(k)} \right)\\
	& =  \sum_{a_{\omega,i}\leq b \leq e_{\omega,i} }
	\left[ f_{\epsilon, \lambda, \omega}(\Delta t, i, b )\log\left(\frac{f_{\epsilon, \lambda, \omega}(\Delta t, i, b )}{\pois\left( \lambda \Delta t; b  - X_{\lfloor\frac{T}{\Delta t}\rfloor\Delta t -(i+1)\Delta t }(\omega)\right)} \right)\right] \\
	& = \left[\eta\left( \left( \frac{\epsilon -L\Delta t}{\epsilon +(1-L)\Delta t}\right)^{d_{\omega}}\right) + \lambda \Delta t\left( \frac{\epsilon -L\Delta t}{\epsilon +(1-L)\Delta t}\right)^{d_{\omega} } \right]\sum_{b=0}^{d_{\omega}} \binom{d_{\omega} }{b}\left( \frac{\Delta t}{\epsilon - L\Delta t}\right)^{b} \\ &\qquad + \left( \left( \frac{\epsilon -L\Delta t}{\epsilon +(1-L)\Delta t}\right)^{d_{\omega}}\right) \sum_{b=0}^{d_{\omega}}\binom{d_{\omega}}{b}\left( \frac{\Delta t}{\epsilon - L\Delta t}\right)^{b}\log\left( \frac{\left(d_{\omega} \right)^{\underline{b}}}{\lambda^{b}(\epsilon - \lambda \Delta t)^{b}}\right)
	\end{align*}
	\end{small}
	where $e_{\omega,i}\in \{a_{\omega,i}, a_{\omega,i}+1 \}$ and $d_{\omega,i}\in \{0,1\}.$
	For any $i = 0,1,\ldots, \tau-1$, if $d_{\omega,i} = 0,$
	then \begin{equation*}
	\begin{aligned}
	KL\left(\mathbb{P}_{X| \overleftarrow{X}, \overleftarrow{Y},i,\Delta t}^{(\omega),(k,l)} \middle| \middle| \mathbb{P}_{X| \overleftarrow{X},i,\Delta t}^{(\omega),(k)} \right) = \lambda \Delta t
	\end{aligned}
	\end{equation*}
	and if $d_{\omega,i} = 1$, then 
	\begin{equation*}
	\begin{aligned}
	&KL\left(\mathbb{P}_{X| \overleftarrow{X}, \overleftarrow{Y},i,\Delta t}^{(\omega),(k,l)} \middle| \middle| \mathbb{P}_{X| \overleftarrow{X},i,\Delta t}^{(\omega),(k)} \right) \\&= \lambda\Delta t\left( \frac{\epsilon -L\Delta t}{\epsilon +(1-L)\Delta t}\right) + \eta\left( \left( \frac{\epsilon -L\Delta t}{\epsilon +(1-L)\Delta t}\right)\right) \\
	&\qquad +\frac{\lambda(\Delta t)^2-\log(\lambda)\Delta t}{\epsilon +(1-L)\Delta t} + \Delta t \eta\left( \frac{1}{\epsilon +(1-L)\Delta t}\right)\\
	& =:S(\lambda, \Delta t).
	\end{aligned}
	\end{equation*}
	Recall that 
	\begin{equation*}
	KL\left(P_{\Delta t}^{(\omega)} \middle|\middle| M_{\Delta t}^{(\omega)}\right)
	= \sum_{i = 0}^{\tau-1}KL\left( \mathbb{P}_{X| \overleftarrow{X}, \overleftarrow{Y},i,\Delta t}^{(\omega),(k,l)} \middle|\middle| \mathbb{P}_{X| \overleftarrow{X},i,\Delta t}^{(\omega),(k)}\right)
	\end{equation*}
	from the proof of Theorem 4 and let $Q_{\omega,\Delta t} = \sum_{i=0}^{\tau-1}d_{\omega,i}.$
	Then $\forall \omega \in \Omega$ we have that
	\begin{align*}\label{S_bound}
	&KL\left( \prod_{i = 0}^{\tau-1} \mathbb{P}_{X| \overleftarrow{X}, \overleftarrow{Y},i,\Delta t}^{(\omega),(k,l)} \middle|\middle|  \prod_{i = 0}^{\tau-1} \mathbb{P}_{X| \overleftarrow{X},i,\Delta t}^{(\omega),(k)}\right)\\
	& =  \sum_{i = 0}^{\tau-1}KL\left( \mathbb{P}_{X| \overleftarrow{X}, \overleftarrow{Y},i,\Delta t}^{(\omega),(k,l)} \middle|\middle|   \mathbb{P}_{X| \overleftarrow{X},i,\Delta t}^{(\omega),(k)}\right) \\
	&= \left(\tau -Q_{\omega,\Delta t}\right)\lambda \Delta t + Q_{\omega, \Delta t}S(\lambda, \Delta t)\\
	& = \lambda \tau \Delta t + Q_{\omega, \Delta t}\left( S(\lambda,\Delta t) - \lambda\Delta t\right)\\
	&\leq \tau S(\lambda, \Delta t).
	\end{align*}
	Since whenever $0 < r <\epsilon$, \begin{equation}\label{ppp_lim}
	\lim_{\Delta t \downarrow 0} \tau S(\lambda, \Delta t)  = (T - t_{0})\left(\lambda - \frac{\log\left( \lambda (\epsilon - r)\right)}{\epsilon - r}\right),
	\end{equation}
	the quantity $KL\left( \prod_{i = 0}^{\tau-1} \mathbb{P}_{X| \overleftarrow{X}, \overleftarrow{Y},i,\Delta t}^{(\omega),(k,l)} \middle|\middle|  \prod_{i = 0}^{\tau-1} \mathbb{P}_{X| \overleftarrow{X},i,\Delta t}^{(\omega),(k)}\right)$ is bounded in a sufficiently small neighborhood of $0$. Note that this limit is independent of the sample path.

	For each $\Delta t>0$ let $A_{\Delta t} = \left\{ \omega \in \Omega :  X_{t+\Delta t}(\omega) - X_{t}(\omega) \leq 1, \forall t\in[t_{0},T) \right\}$ and $B_{\Delta t,\gamma}$ be as in Corollary 5.2; that is,  
	$$
	B_{\Delta t,\gamma} = \left\{\omega\in \Omega : \Delta t' \in (0, \Delta t) \implies  
	KL\left(P_{\Delta t}^{(\omega)} \middle|\middle| M_{\Delta t}^{(\omega)}\right) \leq \gamma\right\}.
	$$
	Fix $\gamma > (T - t_{0})\left(\lambda - \frac{\log\left( \lambda (\epsilon - r)\right)}{\epsilon - r}\right)$.  We have now shown that for all $\Delta t>0$, there exists $0<\widetilde{{\Delta t}}<\Delta t$ such that $A_{\Delta t} \subset B_{\widetilde{\Delta t},\gamma}.$   Furthermore, since $\left( B_{\Delta t,\gamma}\right)_{\Delta t>0}$ is a decreasing collection of sets, 
	\begin{equation}\label{dec_sets}
	\mathbb{P}\left( A_{\Delta t}\right)\leq \mathbb{P}\left( B_{\widetilde{\Delta t},\gamma}\right) \leq \mathbb{P}\left( B_{\Delta t',\gamma}\right) \text{ for all }0 <\Delta t' < \widetilde{\Delta t}.
	\end{equation}
	Due to standard properties of the Poisson point process we have that $\mathbb{P}\left( A_{\Delta t}\right) = 1 - o(\Delta t)$; thus $\mathbb{P}\left( A_{\Delta t}\right)\rightarrow 1$ as $\Delta t\downarrow 0$.  Now (\ref{dec_sets}) yields that $\mathbb{P}\left(B_{\Delta t,\gamma} \right)\rightarrow 1$ as $\Delta t\downarrow 0$, which establishes the existence of processes that satisfy (\ref{PPP_help}) for some $\gamma>0$.

\section{Transfer Entropy Rate}

The generalization of information theoretic measures to the framework of information rates is a common paradigm in information theory.  In this section we address the topic of instantaneous information transfer between processes using our methodology. We begin by defining  transfer entropy rate using the EPT as follows\footnote{A similar definition appears in \cite{spl}.}: 
\begin{definition}
For $t \in [t_{0},T)$, define the {\em transfer entropy rate} from $Y$ to $X$ at $t$, denoted  $\mathbb{T}_{Y \rightarrow X}^{(s,r)}(t)$, by 
\begin{equation}\label{TERate}
\mathbb{T}_{Y \rightarrow X}^{(s,r)}(t)  = \lim_{\Delta t\downarrow 0}\frac{1}{\Delta t}\left({\mathcal{E}\mathcal{P}\mathcal{T}}_{Y \rightarrow X}^{(s,r)}\mid_{t}^{t+\Delta t}\right)
\end{equation}
whenever the limit in (\ref{TERate}) exists.
\end{definition}
\begin{remark}
Suppose the hypotheses of Theorem \ref{main} hold for processes $X$ and $Y$. If $t\in[t_{0},T)$ and $\exists \delta>0$ such that  ${\mathcal{E}\mathcal{P}\mathcal{T}}_{Y \rightarrow X}^{(s,r)}\mid_{t}^{t+dt}<\infty$, for all $dt \in\left(t, t+\delta \right)$, then 
\begin{equation*}
\begin{aligned}\mathbb{T}_{Y \rightarrow X}^{(s,r)}(t) & = \lim_{dt\downarrow 0}\frac{1}{dt}\left({\mathcal{E}\mathcal{P}\mathcal{T}}_{Y \rightarrow X}^{(s,r)}\mid_{t}^{t+ dt}\right)\\
&= \lim_{\substack{dt \downarrow 0 \\ \Delta t \downarrow 0}}\left[ \frac{1}{dt}\sum_{i = 0}^{\left\lfloor \frac{t+dt}{\Delta t}\right\rfloor - \left\lfloor\frac{t}{\Delta t}\right\rfloor -1}       \mathbb{T}_{Y \rightarrow X}^{(k,l),\Delta t}\left(  \left\langle T,i,\Delta t^{\star} \right\rangle\right)\right].
\end{aligned}
\end{equation*}
\end{remark}

Assuming some smoothness of the EPT, we can recover it at any time given the rate by using the following straightforward result.
\begin{lemma}\label{easy1}
If $ [t_{0},T]\ni t \mapsto {\mathcal{E}\mathcal{P}\mathcal{T}}_{Y \rightarrow X}^{(s,r)}\mid_{t_{0}}^{t} \in \mathcal{C}^{1}\left([t_{0},T]\right)$ , then \begin{equation*}
\int_{t_{0}}^{T}\mathbb{T}_{Y \rightarrow X}^{(s,r)}(t) dt = {\mathcal{E}\mathcal{P}\mathcal{T}}_{Y \rightarrow X}^{(s,r)}\mid_{t_{0}}^{T}.
\end{equation*}
\end{lemma}
\begin{proof}
From the fundamental theorem of calculus, we have that
\begin{equation*}
\begin{aligned}
\int_{t_{0}}^{T}\mathbb{T}_{Y \rightarrow X}^{(s,r)}(t) dt &= {\mathcal{E}\mathcal{P}\mathcal{T}}_{Y \rightarrow X}^{(s,r)}\mid_{t_{0}}^{T} - {\mathcal{E}\mathcal{P}\mathcal{T}}_{Y \rightarrow X}^{(s,r)}\mid_{t_{0}}^{t_{0}} \\ &= {\mathcal{E}\mathcal{P}\mathcal{T}}_{Y \rightarrow X}^{(s,r)}\mid_{t_{0}}^{T} - \mathbb{E}_{\mathbb{P}}\left[\log(1)\right]\\
& = {\mathcal{E}\mathcal{P}\mathcal{T}}_{Y \rightarrow X}^{(s,r)}\mid_{t_{0}}^{T} .
\end{aligned}
\end{equation*}
\end{proof}
Note that we have imposed differentiablity in Lemma \ref{easy1}; not just right-hand differentiability.
\begin{lemma}\label{rateiscontant}
Suppose $t_{0}$ and $T$ are distinct elements of $\mathbb{T}$ and $r,s>0$ satisfy $\left( t_{0} - \max{(s,r)}, T\right) \subset \mathbb{T}$. If $Y$ is $(s,r)$-consistent upon $X$ on $[t_{0},T)$ and ${\mathcal{E}\mathcal{P}\mathcal{T}}_{Y \rightarrow X}^{(s,r)}\mid_{t_{0}}^{\cdot}$ is linear on $\left[t_0,T \right]$, then for any $t\in[t_{0},T)$ $$\mathbb{T}_{Y \rightarrow X}^{(s,r)}(t) = \frac{1}{T-t_{0}}{\mathcal{E}\mathcal{P}\mathcal{T}}_{Y \rightarrow X}^{(s,r)}\mid_{t_{0}}^{T}.$$
\end{lemma}
\begin{proof}
  It is immediate that $\mathbb{T}_{Y \rightarrow X}^{(s,r)}$ is constant since ${\mathcal{E}\mathcal{P}\mathcal{T}}_{Y \rightarrow X}^{(s,r)}\mid_{t_{0}}^{\cdot}$ is linear, hence $ {\mathcal{E}\mathcal{P}\mathcal{T}}_{Y \rightarrow X}^{(s,r)}\mid_{t_{0}}^{\cdot} \in \mathcal{C}^{1}\left([t_{0},T]\right)$.  Furthermore, from Lemma \ref{easy1} we have \begin{equation*}
\begin{aligned}
{\mathcal{E}\mathcal{P}\mathcal{T}}_{Y \rightarrow X}^{(s,r)}\mid_{t_{0}}^{T}& = \int_{t_{0}}^{T}\mathbb{T}_{Y \rightarrow X}^{(s,r)}(t') dt'
\\& = \left( T-t_{0}\right)\mathbb{T}_{Y \rightarrow X}^{(s,r)}(t)
\end{aligned}
\end{equation*}
for any $t \in [t_{0},T)$ and the proof is complete.
\end{proof}

\section{Application to stationary processes}

\begin{definition}
Stochastic processes $X$ and $Y$ indexed over $\mathbb{T}$ are {\em conditionally stationary }if  $\forall \omega \in \Omega$  and $ k \geq 1$, all collections of times $\left\{t_{i}\right\}_{0\leq i \leq k}$ in $\mathbb{T}$ such that $t_{i} < t_{i+1}$ for each $i$, and all $A\in \mathcal{X}$,
\begin{equation}
\begin{aligned}
 &\mathbb{P}\left( X_{t_{i+1}} \in A| X_{t_{i}},\ldots X_{t_{i-k}}, Y_{t_{i}}, \ldots Y_{t_{i-k}}\right)(\omega)= \\& \mathbb{P}\left( X_{t_{i+1}+\tau}\in A | X_{t_{i}+\tau},\ldots X_{t_{i-k}+\tau},Y_{t_{i}+\tau}, \ldots, Y_{t_{i-k}+\tau}\right)(\omega)
\end{aligned}
\end{equation}
for all $i \in [k-1]$ and $ \tau >0$.
\end{definition}

\begin{definition}
Suppose $k$ and $l$ are positive integers.  Stochastic processes $X$ and $Y$ on $\mathbb{T}$ are {\em $(k,l)$-order conditionally stationary processes} if $\forall \omega \in \Omega$, all collections of times $\left\{t_{i}\right\}_{0\leq i \leq \max{(k,l)}}$ of $\mathbb{T}$ such that $t_{i} < t_{i+1}$ for each $i$, and all $A\in \mathcal{X}$,
\begin{equation} \label{cond_stat_1}
\begin{aligned}
 &\mathbb{P}\left( X_{t_{i+1}} \in A| X_{t_{i}},\ldots X_{t_{i-k}}, Y_{t_{i}}, \ldots Y_{t_{i-l}}\right)(\omega)= \\&= \mathbb{P}\left( X_{t_{i+1}+\tau}\in A | X_{t_{i}+\tau},\ldots X_{t_{i-k}+\tau},Y_{t_{i}+\tau}, \ldots, Y_{t_{i-l}+\tau}\right)(\omega)
\end{aligned}
\end{equation}
for all $i \in [\max{(k,l)}-1]$ and $ \tau >0$.
 \end{definition}
Observe that if $X$ and $Y$ are conditionally stationary processes, then they are by definition $(k,l)$-order conditionally stationary for all $k,l\geq 1.$  Moreover, if $X$ and $Y$ are stationary, then $\forall \Delta t >0$ and $s,r>0$ such that $[t_{0}-\max(s,r),T) \subset \mathbb{T}$, we have that $X$ and $Y$ are also $\left(\lfloor \frac{s}{\Delta t}\rfloor,\lfloor \frac{r}{\Delta t}\rfloor\right)$-order conditionally stationary.  We exploit this stationarity in the following observation. 
\begin{Observation}\label{stat_obs}
If $X$ and $Y$ are stationary processes, then for any $\Delta t>0$ and $j = 0,\cdots, \tau-1$ we have that
 \begin{equation}\label{stat_TE}
 \begin{aligned}
\sum_{i = 0}^{\tau-1} \mathbb{T}_{Y \rightarrow X}^{(k,l),\Delta t}\left(  \left\langle T,i,\Delta t \right\rangle\right) 
& = \mathbb{E}_{\mathbb{P}}\left[KL\left( \prod_{i = 0}^{\tau-1}\mathbb{P}_{X\mid \overleftarrow{X},\overleftarrow{Y},i,\Delta t}^{(\omega),(\left\lfloor \frac{s}{\Delta t}\right\rfloor,\left\lfloor \frac{r}{\Delta t}\right\rfloor)}  \Bigg| \Bigg| \prod_{i = 0}^{\tau-1}\mathbb{P}_{X\mid \overleftarrow{X},i,\Delta t}^{(\omega),(\lfloor \frac{s}{\Delta t}\rfloor)} \right)\right]\\
& = \tau\mathbb{E}_{\mathbb{P}}\left[KL\left( \mathbb{P}_{X\mid \overleftarrow{X},\overleftarrow{Y},j,\Delta t}^{(\omega),(\lfloor \frac{s}{\Delta t}\rfloor,\lfloor \frac{r}{\Delta t}\rfloor)} \Bigg|\Bigg| \mathbb{P}_{X\mid \overleftarrow{X},j,\Delta t}^{(\omega),(\lfloor \frac{s}{\Delta t}\rfloor)}\right)\right]
\\& = \tau\mathbb{T}_{Y \rightarrow X}^{\left(\left\lfloor \frac{s}{\Delta t}\right\rfloor,\left\lfloor \frac{r}{\Delta t}\right\rfloor\right), \Delta t}\left(  \left\langle T,j,\Delta t \right\rangle\right)
 \end{aligned}
 \end{equation}
where in the second to last equality we used that 
$$
\frac{d\left( c \mathbb{P}_{X\mid \overleftarrow{X},\overleftarrow{Y},j,\Delta t}^{(\omega),(\lfloor \frac{s}{\Delta t}\rfloor,\lfloor \frac{r}{\Delta t}\rfloor)}\right)}{d \left(c \mathbb{P}_{X\mid \overleftarrow{X},j,\Delta t}^{(\omega),(\lfloor \frac{s}{\Delta t}\rfloor)}\right)} = \frac{d \mathbb{P}_{X\mid \overleftarrow{X},\overleftarrow{Y},j,\Delta t}^{(\omega),(\lfloor \frac{s}{\Delta t}\rfloor,\lfloor \frac{r}{\Delta t}\rfloor)}}{d\mathbb{P}_{X\mid \overleftarrow{X},j,\Delta t}^{(\omega),(\lfloor \frac{s}{\Delta t}\rfloor)}},\, \mathbb{P}_{X\mid \overleftarrow{X},j,\Delta t}^{(\omega),(\lfloor \frac{s}{\Delta t}\rfloor)}-\textrm{a.s.}
$$ 
for any $c \neq 0$ due to the a.s.~uniqueness of the RN-derivative.
  \end{Observation}
  
 We can use Observation \ref{stat_obs} to provide an expression for the transfer entropy rate for stationary processes that have $\left(s,r \right)$-consistency on subintervals of $[t_{0}, T)$ of the form $[t_{0},t)$.  It should be noted that a result similar to the statement in part 2 of the following corollary appears as a remark in \cite{spl} without proof.
 \begin{corollary}\label{stationary_monster}
Suppose $\mathbb{T}$ is a closed and bounded interval, $[t_{0},T)\subset \mathbb{T}$, and $r,s>0$ satisfy $\left( t_{0} - \max{(s,r)}, T\right)\subset \mathbb{T}$.  Suppose further that $X$ and $Y$ are stationary processes such that 
\begin{itemize}
    \item[a.]$Y$ is $(s,r)$-consistent upon $X$ on $[t_{0}, t), \forall t \in (t_{0}, T].$
    \item[b.] For all $\forall t \in (t_0, T]$, $\exists M, \delta_{2}>0\text{ such that }\forall \Delta t \in (0, \delta_{2}),$
    $$
    KL\left( \prod_{i = 0}^{\left\lfloor \frac{t}{\Delta t}\right \rfloor - \left\lfloor \frac{t_{0}}{\Delta t}\right \rfloor-1} \mathbb{P}_{X| \overleftarrow{X}, \overleftarrow{Y},i,\Delta t}^{(\omega),(k,l)}\middle|\middle| \prod_{i = 0}^{\left\lfloor \frac{t}{\Delta t}\right \rfloor - \left\lfloor \frac{t_{0}}{\Delta t}\right \rfloor-1} \mathbb{P}_{X| \overleftarrow{X}, i,\Delta t}^{(\omega),(k)}\right)\leq M, \mathbb{P}-\textrm{a.s.}
    $$
\end{itemize}
where $k = \left\lfloor \frac{s}{\Delta t}\right\rfloor$ and $l = \left\lfloor \frac{r}{\Delta t}\right\rfloor$.
 \begin{itemize}
 \item[1.] If $\forall t \in (t_{0}, T]$, $ \lim_{\Delta t \downarrow 0}\frac{1}{\Delta t} \mathbb{T}_{Y \rightarrow X}^{(k,l), \Delta t}\left(\Delta t \left \lfloor \frac{t_{1}}{\Delta t}\right\rfloor\right)$ exists $\forall t_{1}\in [t_{0},t),$ then 
 $$
 \lim_{\Delta t \downarrow 0}\frac{1}{\Delta t} \mathbb{T}_{Y \rightarrow X}^{(k,l), \Delta t}\left( \Delta t \left \lfloor \frac{t_{1}}{\Delta t}\right\rfloor\right) = \frac{{\mathcal{E}\mathcal{P}\mathcal{T}}_{Y \rightarrow X}^{(s,r)}\mid_{t_{0}}^{t_{1}}}{t_{1} - t_{0}}
 $$
 for all $ t_{1}\in (t_{0},t)$.
 \item[2.]$\mathbb{T}_{Y \rightarrow X}^{(s,r)}(t) = \frac{1}{T-t_{0}}{\mathcal{E}\mathcal{P}\mathcal{T}}_{Y \rightarrow X}^{(s,r)}\mid_{t_{0}}^{T}.$
 \end{itemize}
\end{corollary}
\begin{proof}
\textit{(Proof of 1.)} 
Suppose $t\in (t_0,T]$ and $t_{1}\in (t_{0},t)$.  Per assumption $ \lim_{\Delta t \downarrow 0} \mathbb{T}_{Y \rightarrow X}^{(k,l), \Delta t}\left(\Delta t \left \lfloor \frac{t_{1}}{\Delta t}\right\rfloor\right)/\Delta t$ exists, thus we have that
\begin{equation}\label{lim_prod}
\lim_{\Delta t \downarrow 0} \mathbb{T}_{Y \rightarrow X}^{(k,l), \Delta t}\left(\Delta t \left \lfloor \frac{t_{1}}{\Delta t}\right\rfloor\right) = \left( \lim_{\Delta t \downarrow 0}\Delta t\right)\left(  \lim_{\Delta t \downarrow 0}\frac{1}{\Delta t} \mathbb{T}_{Y \rightarrow X}^{(k,l), \Delta t}\left(\Delta t \left \lfloor \frac{t_{1}}{\Delta t}\right\rfloor\right)\right) = 0.
\end{equation}
From  Theorem \ref{main} and (\ref{stat_TE}) we have that 
\begin{equation}\label{t_1_one}
\begin{aligned}
\infty & > {\mathcal{E}\mathcal{P}\mathcal{T}}_{Y \rightarrow X}^{(s,r)}\mid_{t_{0}}^{t_{1}} \\ 
&= \lim_{\Delta t \downarrow 0} \sum_{i = 0}^{\lfloor\frac{t_{1}}{\Delta t}\rfloor - \lfloor\frac{t_{0}}{\Delta t}\rfloor-1}  \mathbb{T}_{Y \rightarrow X}^{(k,l),\Delta t}\left(   \Delta t \left \lfloor \frac{t_{1}}{\Delta t}\right\rfloor-i \Delta t\right)\\
& = \lim_{\Delta t \downarrow 0}\left( \left \lfloor\frac{t_{1}}{\Delta t}\right\rfloor- \left\lfloor \frac{t_{0}}{\Delta t}\right\rfloor \right)\mathbb{T}_{Y \rightarrow X}^{(k,l),\Delta t}\left(   \Delta t \left \lfloor \frac{t_{1}}{\Delta t}\right\rfloor-j \Delta t\right)
\end{aligned}
\end{equation}
for any $j = 0,\cdots, \left\lfloor \frac{t_{1}}{\Delta t}\right\rfloor- \left\lfloor\frac{t_{0}}{\Delta t}\right\rfloor-1$.
Note that for each $\Delta t >0$,  $\exists C_{\Delta t} \in (-2,2)$ such that $$\left \lfloor\frac{t_{1}}{\Delta t}\right\rfloor- \left\lfloor \frac{t_{0}}{\Delta t}\right\rfloor  =  \frac{t_{1} - t_{0}}{\Delta t} + C_{\Delta t}. $$  Letting $j=0$ in (\ref{t_1_one}) we get that
\begin{equation}
\begin{aligned}
& \lim_{\Delta t \downarrow 0}\left( \left\lfloor\frac{t_{1}}{\Delta t}\right\rfloor- \left\lfloor\frac{t_{0}}{\Delta t}\right \rfloor \right)\mathbb{T}_{Y \rightarrow X}^{(k,l),\Delta t}\left( \Delta t \left \lfloor \frac{t_{1}}{\Delta t}\right\rfloor \right) \\
 &= \lim_{\Delta t \downarrow 0}\left( \frac{t_{1} - t_{0}}{\Delta t} +C_{\Delta t} \right)\mathbb{T}_{Y \rightarrow X}^{(k,l),\Delta t}\left( \Delta t \left \lfloor \frac{t_{1}}{\Delta t}\right\rfloor \right)
 \\
 & = (t_{1}-t_{0})\lim_{\Delta t \downarrow 0}\frac{1}{\Delta t}\mathbb{T}_{Y \rightarrow X}^{(k,l),\Delta t}\left( \Delta t \left \lfloor \frac{t_{1}}{\Delta t} \right\rfloor \right) +  \lim_{\Delta t \downarrow 0} C_{\Delta t}\mathbb{T}_{Y \rightarrow X}^{(k,l),\Delta t}\left(\Delta t \left \lfloor \frac{t_{1}}{\Delta t}\right\rfloor \right).
\end{aligned}
\end{equation}
 Since $C_{\Delta t}$ is bounded, $ \lim_{\Delta t \downarrow 0} C_{\Delta t}\mathbb{T}_{Y \rightarrow X}^{(k,l),\Delta t}\left(\Delta t \left \lfloor \frac{t_{1}}{\Delta t}\right\rfloor \right) = 0.$ Now using (\ref{t_1_one}) we get 
 \begin{equation*}
 \begin{aligned}
 (t_{1} - t_{0}) \lim_{\Delta t \downarrow 0}\frac{1}{\Delta t} \mathbb{T}_{Y \rightarrow X}^{(k,l), \Delta t}\left( \Delta t \left \lfloor \frac{t_{1}}{\Delta t}\right\rfloor\right)  =  {\mathcal{E}\mathcal{P}\mathcal{T}}_{Y \rightarrow X}^{(s,r)}\mid_{t_{0}}^{t_{1}}
 \end{aligned}
 \end{equation*}
 and the result follows from division by $t_{1}-t_{0}$.

\textit{(Proof of 2.)}
Suppose $t_{1}, t_{2}$ are distinct elements of $[t_{0}, T]$.  Without loss of generality, suppose $t_{1} > t_{2}\neq t_0$.  Per assumption $X$ and $Y$ are stationary processes such that $Y$ is $(s,r)$-consistent upon $X$ on $[t_{0}, t_{1})$ and $[t_{0}, t_{2})$.  If $j'  = \left\lfloor\frac{t_{1}}{\Delta t}\right\rfloor- \left\lfloor\frac{t_{2}}{\Delta t}\right \rfloor,$ then from (\ref{stat_TE}) we have that
\begin{align*}
& {\mathcal{E}\mathcal{P}\mathcal{T}}_{Y \rightarrow X}^{(s,r)}\mid_{t_{0}}^{t_{1}}\\
& =  \lim_{\Delta t \downarrow 0}\left( \left\lfloor\frac{t_{1}}{\Delta t}\right\rfloor- \left\lfloor\frac{t_{0}}{\Delta t}\right \rfloor \right)\mathbb{T}_{Y \rightarrow X}^{(k,l),\Delta t}\left( \Delta t \left \lfloor \frac{t_{1}}{\Delta t} \right\rfloor - j'\Delta t \right) \\
 &= \lim_{\Delta t \downarrow 0}\left( \frac{t_{1} - t_{0}}{\Delta t} +C_{\Delta t} \right)\mathbb{T}_{Y \rightarrow X}^{(k,l),\Delta t}\left( \Delta t \left \lfloor \frac{t_{2}}{\Delta t}\right\rfloor \right)
 \\& = \lim_{\Delta t \downarrow 0}\frac{t_{1}-t_{0}}{t_{2} - t_{0}}\left( \frac{t_{1} - t_{0} + \Delta t C_{\Delta t}}{(t_{1} - t_{0})\Delta t}\right)(t_{2} - t_{0})\mathbb{T}_{Y \rightarrow X}^{(k,l),\Delta t}\left( \Delta t \left \lfloor \frac{t_{2}}{\Delta t}\right\rfloor \right) 
 \\& =  \frac{t_{1}-t_{0}}{t_{2} - t_{0}}\lim_{\Delta t \downarrow 0}\left(\frac{ \Delta t C_{\Delta t}}{(t_{1} - t_{0})}\right)\left(\left \lfloor \frac{t_{2}}{\Delta t}\right\rfloor  - \left \lfloor \frac{t_{0}}{\Delta t}\right\rfloor - K_{\Delta t}\right)\mathbb{T}_{Y \rightarrow X}^{(k,l),\Delta t}\left( \Delta t \left \lfloor \frac{t_{2}}{\Delta t}\right\rfloor \right) \\
 & + \frac{t_{1}-t_{0}}{t_{2} - t_{0}}\lim_{\Delta t \downarrow 0}\left(\left \lfloor \frac{t_{2}}{\Delta t}\right\rfloor  - \left \lfloor \frac{t_{0}}{\Delta t}\right\rfloor - K_{\Delta t}\right)\mathbb{T}_{Y \rightarrow X}^{(k,l),\Delta t}\left( \Delta t \left \lfloor \frac{t_{2}}{\Delta t}\right\rfloor \right).
 \end{align*}
 Per assumption, $ \lim_{\Delta t\downarrow 0}\left(\left \lfloor \frac{t_{2}}{\Delta t}\right\rfloor  - \left \lfloor \frac{t_{0}}{\Delta t}\right\rfloor\right)\mathbb{T}_{Y \rightarrow X}^{(k,l),\Delta t}\left( \Delta t \left \lfloor \frac{t_{2}}{\Delta t}\right\rfloor \right)$ exists and since both $C_{\Delta t}$ and $K_{\Delta t}$ are bounded we have 
 $$
 \frac{t_{1}-t_{0}}{t_{2} - t_{0}}\lim_{\Delta t \downarrow 0}\left(\frac{ \Delta t C_{\Delta t}}{(t_{1} - t_{0})}\right)\left(\left \lfloor \frac{t_{2}}{\Delta t}\right\rfloor  - \left \lfloor \frac{t_{0}}{\Delta t}\right\rfloor \right)\mathbb{T}_{Y \rightarrow X}^{(k,l),\Delta t}\left( \Delta t \left \lfloor \frac{t_{2}}{\Delta t}\right\rfloor \right) = 0
 $$
and 
$$
\frac{t_{1}-t_{0}}{t_{2} - t_{0}}\lim_{\Delta t \downarrow 0}\left(\frac{ \Delta t C_{\Delta t}}{(t_{1} - t_{0})}\right)K_{\Delta t}\mathbb{T}_{Y \rightarrow X}^{(k,l),\Delta t}\left( \Delta t \left \lfloor \frac{t_{2}}{\Delta t}\right\rfloor \right) = 0. 
$$
Moreover,
\begin{equation*}
\begin{aligned}
& {\mathcal{E}\mathcal{P}\mathcal{T}}_{Y \downarrow X}^{(s,r)}\mid_{t_{0}}^{t_{1}}\\
& =\frac{t_{1}-t_{0}}{t_{2} - t_{0}}\lim_{\Delta t \downarrow 0}\left(\frac{ \Delta t C_{\Delta t}}{(t_{1} - t_{0})}\right)\left(\left \lfloor \frac{t_{2}}{\Delta t}\right\rfloor  - \left \lfloor \frac{t_{0}}{\Delta t}\right\rfloor - K_{\Delta t}\right)\mathbb{T}_{Y \rightarrow X}^{(k,l),\Delta t}\left( \Delta t \left \lfloor \frac{t_{2}}{\Delta t}\right\rfloor \right) \\
&+ \frac{t_{1}-t_{0}}{t_{2} - t_{0}}\lim_{\Delta t \downarrow 0}\left(\left \lfloor \frac{t_{2}}{\Delta t}\right\rfloor  - \left \lfloor \frac{t_{0}}{\Delta t}\right\rfloor - K_{\Delta t}\right)\mathbb{T}_{Y \rightarrow X}^{(k,l),\Delta t}\left( \Delta t \left \lfloor \frac{t_{2}}{\Delta t}\right\rfloor \right)\\
& =  \frac{t_{1}-t_{0}}{t_{2} - t_{0}}\lim_{\Delta t \downarrow 0}\left(\left \lfloor \frac{t_{2}}{\Delta t}\right\rfloor  - \left \lfloor \frac{t_{0}}{\Delta t}\right\rfloor - K_{\Delta t}\right)\mathbb{T}_{Y \rightarrow X}^{(k,l),\Delta t}\left( \Delta t \left \lfloor \frac{t_{2}}{\Delta t}\right\rfloor \right)\\
\end{aligned}
\end{equation*}
and since $\frac{t_{1}-t_{0}}{t_{2} - t_{0}}\lim_{\Delta t \downarrow 0} K_{\Delta t}\mathbb{T}_{Y \rightarrow X}^{(k,l),\Delta t}\left( \Delta t \left \lfloor \frac{t_{2}}{\Delta t}\right\rfloor \right) = 0$, we have 
\begin{equation*}
\begin{aligned}
&{\mathcal{E}\mathcal{P}\mathcal{T}}_{Y \rightarrow X}^{(s,r)}\mid_{t_{0}}^{t_{1}} = \frac{t_{1}-t_{0}}{t_{2} - t_{0}}\lim_{\Delta t \downarrow 0}\left(\left \lfloor \frac{t_{2}}{\Delta t}\right\rfloor  - \left \lfloor \frac{t_{0}}{\Delta t}\right\rfloor\right)\mathbb{T}_{Y \rightarrow X}^{(k,l),\Delta t}\left( \Delta t \left \lfloor \frac{t_{2}}{\Delta t}\right\rfloor \right)
\\& \implies {\mathcal{E}\mathcal{P}\mathcal{T}}_{Y \rightarrow X}^{(s,r)}\mid_{t_{0}}^{t_{2}} = \frac{t_{2} - t_{0}}{t_{1}-t_{0}}{\mathcal{E}\mathcal{P}\mathcal{T}}_{Y \rightarrow X}^{(s,r)}\mid_{t_{0}}^{t_{1}}.
\end{aligned}
\end{equation*}
Thus, ${\mathcal{E}\mathcal{P}\mathcal{T}}_{Y \rightarrow X}^{(s,r)}\mid_{t_{0}}^{t}$ is linear in $t -t_{0}$ and the result follows immediately from Lemma \ref{rateiscontant}. 
\end{proof}
Simply put, Corollary \ref{stationary_monster} states that under stationarity in a rather strict sense, the TE rate is the average value of the expected pathwise transfer entropy.

\section{Jump Processes}\label{cadlag_sec}

In this section we consider EPT between jump  processes, i.e., processes whose sample paths, with probability one, are step functions.  These processes are ubiquitous in the literature concerning the application of TE to neural spike trains, social media sentiment analysis, and similar fields.  Examples of such processes are L\'{e}vy processes and Poisson processes.  Furthermore, we define conditional escape and transition rates similar to those in \cite{spl} as follows.
\begin{definition}
For jump processes $X=\left(X_{t}\right)_{t\in[t_{0},T)}$ and  $Y=\left(Y_{t}\right)_{t\in[t_{0},T)}$ with $\Sigma$ countable, define for each $\omega \in \Omega, t \in [t_{0},T); r,s>0$, and $x'\in\Sigma$ the {\em conditional transition rate of} $X$ {\em given} $X$ {\em and} $Y$ {\em of} $x'$ {\em at} $t$,  denoted $\psi\left[ x' \middle| \overleftarrow{X},\overleftarrow{Y}\right](t,\omega)$, by
\begin{equation}\label{trans_xy}
\begin{aligned}
&\hspace{2in}\psi\left[ x' \middle| \overleftarrow{X},\overleftarrow{Y}\right](t,\omega)=
\\ &
 \lim_{\Delta t \downarrow 0}\frac{1}{\Delta t} \mathbb{P}\left(\{\omega'\in\Omega : \exists t'\in[t,t+\Delta t) \text{ s.t. } X_{t'}(\omega') = x' \}\mid X_{t^{-}-s}^{t^{-}}, Y_{t^{-}-r}^{t^{-}} \right)\left(\omega\right),
\end{aligned}
\end{equation}
the {\em conditional transition rate of} $X$ {\em given }$X$ {\em of} $x'$ {\em at} $t$, denoted $\psi\left[ x' \middle| \overleftarrow{X}\right](t,\omega)$, by
 \begin{equation}\label{trans_x}
 \begin{aligned}
& \hspace{1.5 in}\psi\left[ x' \middle|   \overleftarrow{X}\right](t,\omega)= \\
& \lim_{\Delta t \downarrow 0}\frac{1}{\Delta t} \mathbb{P}\left(\left\{\omega'\in\Omega : \exists  t'\in[t,t+\Delta t)\text{ s.t. } X_{t'}(\omega') = x' \right\}\mid X_{t^{-}-s}^{t^{-}} \right)\left(\omega\right),\\
\end{aligned}
\end{equation}
and the {\em conditional escape rates} $\lambda^{(s)}_{X | X}(t,\omega)$ and $ \lambda^{(s,r)}_{X | X,Y}(t,\omega)$ by \begin{equation}\label{esc_X}
\lambda^{(s)}_{X | X}(t,\omega) = \sum_{x'\in\Sigma,x'\neq x_{t}^{-}}\psi\left[ x'\middle| \overleftarrow{X}\right](t,\omega)
\end{equation}
and
\begin{equation}\label{esc_XY}
\lambda^{(s,r)}_{X\mid X,Y}(t,\omega) = \sum_{x'\in\Sigma,x'\neq x_{t}^{-}}\psi\left[x' \middle| \overleftarrow{X},\overleftarrow{Y} \right](t,\omega).
\end{equation}
\end{definition}
\begin{remark}
In the forthcoming, we will sometimes regard the conditional transition rates defined above as measures on the space $\left(\Sigma, \mathcal{X} \right)$ for fixed $ \omega \in \Omega, t \in \mathbb{T}$ in accordance with standard definitions of transition kernels (see Section 1.2 of \cite{jacobsen}). 
\end{remark}
\begin{Notation}
 for $t\in[t_{0},T), \omega\in\Omega, $and $s,r>0$, let $$\Delta \lambda^{(s,r)}(t,\omega) = \lambda^{(s)}_{X\mid X}(t,\omega) - 
\lambda^{(s,r)}_{X\mid X,Y}(t,\omega).$$
\end{Notation}
 We now consider TE between time-homogeneous Markov processes. 
 
 \begin{definition}
 Suppose $\left( \Omega, \mathcal{F}, \mathbb{P}\right)$ is a probability space, $\mathbb{T}\subset \mathbb{R}_{\geq 0}$ is a bounded and closed interval, $\Sigma$ is a countable set, and $\mathcal{X}$ is a $\sigma-$algebra of subsets of $\Sigma$ containing all singletons of $\Sigma$.  A stochastic process $X = \left( X_{t}\right)_{t\in \mathbb{T}}$ is a {\em time-homogeneous Markov jump process} if all of its sample paths are piecewise constant and right-continuous and $\forall n\geq 1$, times $t_{0}<t_{1}<\dots<t_{n-1}$, and sets $A_{i}\in\mathcal{X}$ for all $ 0\leq i\leq n$, \begin{equation*}
 \begin{aligned}
 &\mathbb{P}_{t_{n-1}+\tau}\left[X_{t_{n-1}+\tau}\in A_{n-1}\middle| X_{t_{n-2}+\tau}, \cdots, X_{t_{0}+\tau}\right](\omega)\\
 &=\mathbb{P}_{t_{n-1}+\tau}\left[X_{t_{n-1}+\tau}\in A_{n-1}\middle|  X_{t_{n-2}+\tau}\right](\omega)\\
 &= \mathbb{P}_{t_{n-1}}\left[X_{t_{n-1}}\in A_{n-1} \middle|X_{t_{n-2}}\right](\omega)
 \end{aligned}
 \end{equation*}
 for each $\omega\in\Omega$ and all $\tau \geq 0$ such that $t_{i-1+\tau}\in\mathbb{T}$ for all $0\leq i\leq n$. 
 \end{definition}
 We now present a Girsanov formula for the pathwise transfer entropy when the destination process is a time-homogeneous Markov jump process and the source process is any jump process.
\begin{theorem}\label{markov_theorem}
Suppose $\Sigma$ is countable.  Suppose further that $X$ and $Y$ are jump stochastic processes on $\mathbb{T}$ with $[t_{0},T)\subset \mathbb{T}$ and $X$ is a time-homogeneous Markov process with conditional transition rates given by (\ref{trans_xy}) and (\ref{trans_x}) and conditional escape rates given by (\ref{esc_XY}) and (\ref{esc_X}).  If 
\begin{itemize}
\item[1.] $\forall \omega \in \Omega$, $\psi\left[ x_{t_0}\middle| \overleftarrow{X},\overleftarrow{Y}\right](t_{0},\omega) = \psi\left[ x_{t_0}\middle| \overleftarrow{X}\right](t_{0},\omega).$
\item[2.] The conditional escape rates are bounded and positive.
\item[3.] $\psi\left[ \cdot \middle| \overleftarrow{X}, \overleftarrow{Y}\right](t,\omega)    \ll  \psi\left[ \cdot \middle| \overleftarrow{X}\right](t,\omega)$ for each $\omega \in \Omega$ and $t \in [t_{0},T).$
\end{itemize}
Then 
\begin{equation}
\begin{aligned}
{\mathcal{P}\mathcal{T}}_{Y \rightarrow X}^{(s,r)}\mid_{t_0}^{T}\left( \omega, x_{t_{0}}^{T} \right)= &\sum_{i=1}^{N_{X}^{[t_{0},T)}\left(x_{t_{0}}^{T}\right)}\log\left[\frac{\psi\left[ x_{\tau_i}\middle| \overleftarrow{X},\overleftarrow{Y}\right](\tau_{i},\omega)}{\psi\left[ x_{\tau_i}\middle| \overleftarrow{X}\right](\tau_{i},\omega)}\right] \\
& \qquad +\int_{t_{0}}^{T}\left( \Delta \lambda^{(s,r)}(t,\omega)\right)dt
\end{aligned}
\end{equation}
for every $ \omega \in \Omega$ and every sample path $x_{t_{0}}^{T}$ of $X$.
\end{theorem}
\begin{proof}

Since $X$ is Markov, there exists an increasing sequence of finite random jump times $\{\tau_{n}\}_{n \geq 0}$ such that $\tau_{0} =t_{0}$, $X_{\tau_{n}}$ is constant on $[\tau_{n}, \tau_{n+1})$, and $X_{\tau_{n}^{-}} \neq X_{\tau_{n}}$.  Furthermore, from the Markov assumption, conditionally on $\left\{ X_{\tau_{n}}\right\}_{n\geq 0}$, the variables $\left\{\tau_{n+1}-\tau_{n} \right\}_{n\geq 0}$ are independent and exponentially distributed.
We first need to show that for arbitrary measures $P \ll Q$ on the path space of piecewise constant sample paths of $X$ with transition probabilities $p_{P}(\cdot,\cdot)$,  $p_{Q}(\cdot,\cdot)$ and escape rates $\gamma_{P}, \gamma_{Q}$, that for every realization $x_{t_{0}}^{T}$ of the process $X_{t_{0}}^{T}$,
\begin{equation}\label{mark_PTE}
\frac{dP}{dQ}\left(x_{t_{0}}^{T}\right)
= \sum_{i=0}^{N_{X}^{[t_{0},T)}\left( x_{t_{0}}^{T}\right)}\log\frac{\gamma_{P}(x_{\tau_{i}}^{-})p_{P}\left(x_{\tau{i}}^{-},x_{\tau{i}}\right)}{\gamma_{Q}(x_{\tau_{i}}^{-})p_{Q}\left(x_{\tau{i}}^{-},x_{\tau_{i}}\right)} + \int_{t_{0}}^{T}\left( \gamma_{Q}(x_t) - 
\gamma_{P}(x_{t}^{-})\right)dt
\end{equation}
where $\left\{ \tau_{i}\right\}_{i = 0}^{N_{X}^{[t_{0},T)}}$ is the sequence of jump times of the realization $x_{t_{0}}^{T}$.
A proof of (\ref{mark_PTE}) is given in Appendix 1, Proposition 2.6 of \cite{Kipnis}.
Now letting $P$ and $Q$ be the measures in (\ref{trans_xy}) and (\ref{trans_x}), respectively, using assumption 1., and noting that $$\frac{\psi\left[ x_{\tau_i}\middle| \overleftarrow{X},\overleftarrow{Y}\right](\tau_{i},\omega)}{\lambda^{(s,r)}_{X\mid X,Y}(\tau_{i},\omega)} = p_{X\mid X,Y}(x_{\tau_{i}}, x_{\tau_{i}^{-}},Y_{\tau_{i}^{-}}(\omega))$$
and
$$\frac{\psi\left[ x_{\tau_i}\middle| \overleftarrow{X}\right](\tau_{i},\omega)}{\lambda^{(s)}_{X\mid X}(\tau_{i},\omega)} = p_{X\mid X}(x_{\tau_{i}},x_{\tau_{i}^{-}})$$
where $p_{X\mid X, Y}$ and $p_{X\mid X}$ denote conditional transition probabilities, we get that  
\begin{align*}
&{\mathcal{P}\mathcal{T}}_{Y \rightarrow X}^{(s,r)}\mid_{t_0}^{T}\left( \omega, x_{t_{0}}^{T} \right)=\\& \sum_{i=0}^{N_{X}^{[t_{0},T) \left( x_{t_{0}}^{T}\right)}}\log\left[\frac{\left(   \lambda^{(s,r)}_{X\mid X,Y}(\tau_{i},\omega)  \right) \left(   p_{X\mid X,Y}(x_{\tau_{i}}, x_{\tau_{i}^{-}},Y_{\tau_{i}^{-}}(\omega))   \right)}{\left(  \lambda^{(s)}_{X\mid X}(\tau_{i},\omega)\right)  \left( p_{X\mid X}(x_{\tau_{i}},x_{\tau_{i}^{-}}) \right)}\right]
\\ & \qquad+ \int_{t_{0}}^{T}\left( \Delta \lambda^{(s,r)}(t,\omega)\right)dt\\
 &= \sum_{i=0}^{N_{X}^{[t_{0},T) \left( x_{t_{0}}^{T}\right)}}\log\left[\frac{\psi\left[ x_{\tau_i}\middle| \overleftarrow{X},\overleftarrow{Y}\right](\tau_{i},\omega)}{\psi\left[ x_{\tau_i}\middle| \overleftarrow{X}\right](\tau_{i},\omega)}\right] + \int_{t_{0}}^{T}\left(\Delta \lambda^{(s,r)}(t,\omega)\right)dt \\
& =  \log\left[\frac{\psi\left[ x_{0}\middle| \overleftarrow{X},\overleftarrow{Y}\right](\tau_{0},\omega)}{\psi\left[ x_{0}\middle| \overleftarrow{X}\right](\tau_{0},\omega)}\right]+\sum_{i=1}^{N_{X}^{[t_{0},T)\left(x_{t_{0}}^{T} \right)}}\log\left[\frac{\psi\left[ x_{\tau_i}\middle| \overleftarrow{X},\overleftarrow{Y}\right](\tau_{i},\omega)}{\psi\left[ x_{\tau_i}\middle| \overleftarrow{X}\right](\tau_{i},\omega)}\right] \\
& \qquad + \int_{t_{0}}^{T}\left(\Delta \lambda^{(s,r)}(t,\omega)\right)dt \\
&=\sum_{i=1}^{N_{X}^{[t_{0},T)}(x_{t_{0}}^{T})}\log\left[\frac{\psi\left[ x_{\tau_i}\middle| \overleftarrow{X},\overleftarrow{Y}\right](\tau_{i},\omega)}{\psi\left[ x_{\tau_i}\middle| \overleftarrow{X}\right](\tau_{i},\omega)}\right]+ \int_{t_{0}}^{T}\left(\Delta \lambda^{(s,r)}(t,\omega)\right)dt .
\end{align*}
\end{proof}

From here, we present the following explicit formula for the TE rate when the source process is a time homogeneous Markov jump process and the destination process is a time homogeneous Poisson process. 
\begin{corollary}
Suppose $X$ is a time homogeneous Poisson process and $Y$ is a time homogeneous Markov jump process on $[t_{0},T)$ such that the hypotheses of Theorem \ref{markov_theorem} hold.  If  $t \mapsto \log\left[\frac{\psi\left[ x_{t}\middle| \overleftarrow{X},\overleftarrow{Y}\right](t,\omega)}{\psi\left[ x_{t}\middle| \overleftarrow{X}\right](t,\omega)}\right] \in L_{1}([t_{0},T), \mu)$ for each $\omega \in \Omega$, then $\forall t \in [t_{0},T)$ the transfer entropy rate, $\mathbb{T}^{(s,r)}_{Y \rightarrow X}(t),$ is given by
\begin{small}
\begin{equation}
\begin{aligned}
&\mathbb{T}^{(s,r)}_{Y \rightarrow X}(t) = \\&  \qquad\mathbb{E}_{\mathbb{P}}\left[ \mathbb{E}_{\mathbb{P}_{X\mid X, Y(\cdot)}^{(s,r)}}\left[ \lambda^{(s,r)}_{X\mid X,Y}(t,\cdot) \left(\log\left[\frac{\psi\left[ x_{t}\middle| \overleftarrow{X},\overleftarrow{Y}\right](t,\cdot)}{\psi\left[ x_{t}\middle| \overleftarrow{X}\right](t,\cdot)}\right]-1\right)  +  \lambda^{(s)}_{X\mid X}(t,\cdot) \right]\right]. 
\end{aligned}
  \end{equation}
  \end{small}
\end{corollary}
\begin{proof}
Observe that for each $\omega \in \Omega$ and sample path $x_{t_0}^{T}$ we have
\begin{align*}
&{\mathcal{P}\mathcal{T}}_{Y \rightarrow X}^{(s,r)}\mid_{t_0}^{T}\left( \omega, x_{t_{0}}^{T} \right)= \\& \sum_{i=1}^{N_{X}^{[t_{0},T)}(\omega)}\log\left[\frac{\psi\left[ x_{\tau_i}\middle| \overleftarrow{X},\overleftarrow{Y}\right](\tau_{i},\omega)}{\psi\left[ x_{\tau_i}\middle| \overleftarrow{X}\right](\tau_{i},\omega)}\right] + \int_{t_{0}}^{T}\left( \Delta \lambda^{(s,r)}(t,\omega)\right)dt \\
& = \int_{t_{0}}^{T}\log\left[\frac{\psi\left[ x_{t}\middle| \overleftarrow{X},\overleftarrow{Y}\right](t,\omega)}{\psi\left[ x_{t}\middle| \overleftarrow{X}\right](t,\omega)}\right]dN_{X}^{[t_{0},t)}(\omega) + \int_{t_{0}}^{T}\left( \Delta \lambda^{(s,r)}(t,\omega)\right)dt.
\end{align*}
Since the process $\left( N_{X}^{[t_{0},t)}(\cdot) - \int_{t_{0}}^{t} \lambda_{X \mid X,Y}(t',\cdot)dt'\right)_{t \in [t_{0},T)}$ is a martingale, the stochastic process $\left(\int_{t_{0}}^{t}\log\left[\frac{\psi\left[ x_{t'}\middle| \overleftarrow{X},\overleftarrow{Y}\right](t',\cdot)}{\psi\left[ x_{t'}\middle| \overleftarrow{X}\right](t',\cdot)}\right]dN_{X}^{[t_{0},t')}(\cdot)\right)_{t\in[t_0,T)}$ is a martingale such that for each $\omega\in\Omega$
\begin{equation}\label{Brem_cons}
\begin{aligned}
&\mathbb{E}_{\mathbb{P}_{X\mid X, Y(\omega)}^{(s,r)}}\left[ \int_{t_{0}}^{T}\log\left[\frac{\psi\left[ x_{t}\middle| \overleftarrow{X},\overleftarrow{Y}\right](t,\cdot)}{\psi\left[ x_{t}\middle| \overleftarrow{X}\right](t,\cdot)}\right]dN_{X}^{[t_{0},t)}(\cdot) \right] \\ &= \mathbb{E}_{\mathbb{P}_{X\mid X, Y(\omega)}^{(s,r)}}\left[ \int_{t_{0}}^{T}\left(\lambda^{(s,r)}_{X\mid X,Y}(t,\cdot)\log\left[\frac{\psi\left[ x_{t}\middle| \overleftarrow{X},\overleftarrow{Y}\right](t,\cdot)}{\psi\left[ x_{t}\middle| \overleftarrow{X}\right](t,\cdot)}\right]\right) dt\right]
\end{aligned}
\end{equation}
as a consequence of Theorem 9.2.1 of \cite{Bremaud}.
Now let $\widetilde{\psi}_{t,\omega} = \psi\left[ x_{t}\middle| \overleftarrow{X},\overleftarrow{Y}\right](t,\omega)$, $\Bar{\psi}_{t,\omega} = \psi\left[ x_{t}\middle| \overleftarrow{X}\right](t,\omega)$, and $f(t,\omega) = \lambda^{(s,r)}_{X\mid X,Y}(t,\omega) \left(\log\left[\frac{\widetilde{\psi}_{t,\omega}}{\Bar{\psi}_{t,\omega}}\right]-1\right)$ for each $t\in[t_{0},T)$ and $\omega\in\Omega.$  From Theorem \ref{markov_theorem} and (\ref{Brem_cons}) we have 
\begin{equation}
\begin{aligned}
&\mathbb{T}^{(s,r)}_{Y \rightarrow X}(t)=   \lim_{\Delta t \downarrow 0} \frac{1}{\Delta t}\mathbb{E}_{\mathbb{P}}\left[
\mathbb{E}_{\mathbb{P}_{X\mid X, Y(\omega)}^{(s,r)}}\left[ \int_{t}^{t+\Delta t}\left[f(t',\omega)  +  \lambda^{(s)}_{X\mid X}(t',\omega)\right]dt' \right]\right]
\\ & = 
\mathbb{E}_{\mathbb{P}}\left[ \mathbb{E}_{\mathbb{P}_{X\mid X, Y(\omega)}^{(s,r)}}\left[  \left(\lambda^{(s,r)}_{X\mid X,Y}(t,\omega)\right) \left(\log\left[\frac{\widetilde{\psi}_{t,\omega}}{\Bar{\psi}_{t,\omega}}\right]-1\right)  +  \lambda^{(s)}_{X\mid X}(t,\omega) \right]\right]
\end{aligned}
\end{equation}
where the last equality follows from Theorem A16.1 in \cite{Williams}.

\end{proof}

\section{Conclusion}
We end with some open problems regarding the present work.  First, motivated by \cite{Wyner}, we present an alternative definition of EPT in which, we define it as a limit superior of conditional mutual information over sub-partitions of the interval $[t_{0}, T)$.   
 We begin by defining sub-partitions of an interval of the form $[t_{0}, T)$. 
\begin{definition}
A {\em sub-partition} $P$ of an interval $[t_{0},T)\subset \mathbb{R}$ is a set of real numbers $t_{0}, t_{1}, \dots, t_{n}$ such that $$t_{0}<t_{1}< \dots< t_{n}<T.$$
\end{definition}
\begin{definition}
Suppose $\mathbb{T}$ is a closed and bounded interval and let $P_{[t_{0},T)}$ denote the set of sub-partitions of the interval $[t_{0},T)\subset \mathbb{T}$ and $\left|\left|P\right|\right|$ denote the {\em mesh} of a sub-partition $P\in P_{[t_{0},T)}$, defined by 
$$
\left|\left|P\right|\right| = \max_{\substack{t_{i}\in P\\ i\geq 1}} \left| t_{i}-t_{i-1}\right|.
$$ 
For all $P \in P_{[t_{0},T)}$; $r,s>0$, such that $(t_{0} - \max{(r,s)},T]\subset \mathbb{T}$, define the {\em sub-partitioned expected pathwise transfer entropy} of the sub-partition $P$, denoted ${\mathcal{E}\mathcal{P}\mathcal{T}}_{Y \rightarrow X}^{(s,r), P}\mid_{t_{0}}^{T}$, by \begin{equation}\label{EPT_1}
{\mathcal{E}\mathcal{P}\mathcal{T}}_{Y \rightarrow X}^{(s,r),P}\mid_{t_{0}}^{T} = \sum_{i=1}^{\left|\left| P\right|\right|} I\left( X_{t_{i-1}}^{t_{i}}; Y_{t_{i}-r}^{t_{i}}\mid X_{t_{i-1}-s}^{t_{i-1}} \right).
\end{equation}
\end{definition}
\begin{definition}\label{new_EPT}
Suppose $\mathbb{T}$ is a closed and bounded interval such that  $[t_{0},T)\subset \mathbb{T}$. For all $r,s>0$ such that $(t_{0} - \max{(r,s)},T]\subset \mathbb{T}$, define 
\begin{equation}
\begin{aligned}\label{alt_def}
\widetilde{\mathcal{E}\mathcal{P}\mathcal{T}}_{Y \rightarrow X}^{(s,r)}\mid_{t_{0}}^{T}&:=\limsup_{\substack{\Delta t \downarrow 0\\ P\in P_{[t_{0},T)}, \left|\left| P\right|\right|\leq \Delta t }}{\mathcal{E}\mathcal{P}\mathcal{T}}_{Y \rightarrow X}^{(s,r),P}\mid_{t_{0}}^{T}\\& = \limsup_{\substack{\Delta t \downarrow 0 \\ P\in P_{[t_{0},T)}, \left|\left| P\right|\right|\leq \Delta t }} \sum_{i=1}^{\left|\left| P\right|\right|} I\left( X_{t_{i-1}}^{t_{i}}; Y_{t_{i}-r}^{t_{i}}\mid X_{t_{i-1}-s}^{t_{i-1}} \right).
\end{aligned}
\end{equation}
\end{definition}

\begin{Question}
Is this definition advantageous or even equivalent to Definition \ref{EPT}?
\end{Question}
\indent In Section \ref{Application: Lagged Poisson point process} we presented an explicit form of $KL\left(P_{\Delta t}^{(\omega)}\middle| \middle|  M_{\Delta t}^{(\omega)} \right)$ and demonstrated that it satisfied sufficient conditions of Corollary \ref{PPPusecase}. We propose the following natural question.
\begin{Question}
What other processes satisfy (\ref{bounder}) or  (\ref{PPP_help}) other than the deterministically lagged counting process of a time homogeneous Poisson point process?
\end{Question}
\indent In the Appendix section, we provide an explicit form for the divergence $KL\left(\mathbb{P}_{X| \overleftarrow{X}, \overleftarrow{Y},i,\Delta t}^{(\omega),(k,l)} \middle| \middle| \mathbb{P}_{X| \overleftarrow{X},i,\Delta t}^{(\omega),(k)} \right)$ where $Y$ is a time-lagged version of a Wiener process $X$. However, there is no explicit form for neither $KL\left(P_{\Delta t}^{(\omega)} \middle| \middle|  M_{\Delta t}^{(\omega)} \right)$ nor $KL\left(\mathbb{P}_{X| \overleftarrow{X}, \overleftarrow{Y},i,\Delta t}^{(\omega),(k,l)} \middle| \middle| \mathbb{P}_{X| \overleftarrow{X},i,\Delta t}^{(\omega),(k)} \right)$ other than those presented in the present work.  There are a myriad of transformations one could perform on a process to yield another, for example, thinning, superimposition, deterministic and random lagging, and convolution.  Each of these transformations yields a new process that is not independent of the original process; thus, in general, there ought to be a nonzero TE between the two.  Compound Poisson processes (CPP) are of particular relevance to the continuous-time framework presented in this work and are widely used to model neural spike trains, social media sentiment, geological activity, etc.; therefore, a demonstration that either (\ref{bounder}) or (\ref{PPP_help}) hold for pairs of processes derived from variously transformed CPPs may be useful for applications. 
 
One of the main contributions of this work is a definition of the TE rate native to continuous-time processes.  However, our methodology does not present any practical means of measuring it.

\begin{Question}
Do there exist practical estimators of the EPT and the TE rate, at least for common process types?
\end{Question}

The transfer entropy estimator presented in \cite{kozachenko1987sample} is of practical utility for discrete-time processes.  Can it be generalized to appropriately measure TE using the measure theoretical approach taken in this work? If so, what are its properties?  There is a wealth of questions one could propose pertaining to such an estimator, e.g., is this estimator biased or asymptotically biased/unbiased?  Is it an efficient estimator and how is its speed performance?  Does there exist an appropriate model class under which an MLE for TE exists?  How does this estimator compare with binning and partitioning based estimators?
     
If there is no such estimator that can be used in a general setting, does there exist an estimator when the destination and source process are a particular type of continuous-time stochastic process?  Providing estimators for TE rate and EPT between a pair of time inhomogeneous PPPs, compound Poisson processes, or Brownian motions with various effects on each other would likely be helpful in understanding a wide variety of linked, real-world time series.


%
%



\end{document}